\documentclass[11pt]{amsart}
\usepackage[utf8]{inputenc}

\usepackage{mystyle}

\begin{document}

\title[Graph and hypergraph colouring via nibble methods: A survey]{Graph and hypergraph colouring via nibble methods:\\
A survey}

\date{November 16, 2021}

\email{\{D.Y.Kang.1, T.J.Kelly, D.Kuhn, D.Osthus\}@bham.ac.uk, abhishekmethuku@gmail.com}
\address{School of Mathematics, University of Birmingham,
Edgbaston, Birmingham, B15 2TT, United Kingdom}

\author[Kang]{Dong Yeap Kang}

\author[Kelly]{Tom Kelly}

\author[K\"uhn]{Daniela K\"uhn}

\author[Methuku]{Abhishek Methuku}

\author[Osthus]{Deryk Osthus}

\thanks{This project has received partial funding from the European Research
Council (ERC) under the European Union's Horizon 2020 research and innovation programme (grant agreement no. 786198, D.~K\"uhn and D.~Osthus).
The research leading to these results was also partially supported by the EPSRC, grant nos. EP/N019504/1 (D.~Kang, T. Kelly and D.~K\"uhn) and EP/S00100X/1 (A.~Methuku and D.~Osthus).}

\begin{abstract}
    This paper provides a survey of methods, results, and open problems on graph and hypergraph colourings, with a particular emphasis on semi-random `nibble' methods.  We also give a detailed sketch of some aspects of the recent proof of the Erd\H os-Faber-Lov\'asz conjecture.
    
\end{abstract}


\maketitle

\section{Introduction}

The theory of graph and hypergraph colouring is fundamental to combinatorics, with 
numerous applications to other areas of combinatorics and beyond.  It has also given rise to the introduction and development of techniques that have had a major impact far beyond the settings for which they were initially developed.   In this paper, we survey results, open problems, and methods in the area, with a focus on one such technique called the `R\"odl nibble' or the `semi-random method'. 
We also provide a detailed outline of some ideas involved in the authors' recent proof of the Erd\H os-Faber-Lov\'asz conjecture~\cite{KKKMO2021}, 
with the R\"odl nibble playing an important role.

\subsection{Background}
A \textit{hypergraph} $\cH$ is a pair $\cH = (V, E)$, where $V$ is a set of elements called \textit{vertices} and $E \subseteq 2^{V}$ is a set of subsets of $V$ called \textit{edges}.  For convenience, we often identify a hypergraph $\cH$ with its edge set and use $V(\cH)$ to denote its vertex set.  A \textit{proper edge-colouring} of a hypergraph $\cH$ is an assignment of colours to the edges of $\cH$ such that no two edges of the same colour share a vertex, and a \textit{proper vertex-colouring} (often simply called a proper colouring) of $\cH$ is an assignment of colours to the vertices of $\cH$ such that no edge contains vertices all of the same colour.  The \textit{chromatic index} of $\cH$, denoted $\chi'(\cH)$, is the minimum number of colours used by a proper edge-colouring of $\cH$, and the \textit{chromatic number}, denoted $\chi(\cH)$, is the minimum number of colours used by a proper vertex-colouring of $\cH$.  
A hypergraph $\cH$ is \textit{$k$-uniform} if every edge $e \in \cH$ satisfies $|e| = k$, and a \textit{graph} is a 2-uniform hypergraph.  A \textit{matching} $M\subseteq \cH$ in a hypergraph $\cH$ is a set of disjoint edges, and an independent set $I \subseteq V(\cH)$ in $\cH$ is a set of vertices that contains no edge of $\cH$ (as a subset).  The maximum size of an independent set in $\cH$, denoted $\alpha(\cH)$, is called the \textit{independence number} of $\cH$.  

Bounding the chromatic index of a graph or hypergraph is closely related to the problem of finding large matchings (note that the colour classes of a proper edge-colouring form a matching).
Matching theory is a classical subject in the study of graphs and is well developed, dating back to the work of K\"onig~\cite{K31}, Egerv\'ary~\cite{E31}, and Hall~\cite{H35}in the 1930s.  Tutte's theorem~\cite{T47} provides a characterization of graphs that contain a perfect matching, and Edmonds'~\cite{E65} `Blossom algorithm' finds a maximum matching in a graph in polynomial time.  In contrast, there is no polynomial-time algorithm known to compute the independence number or chromatic number of a graph, or the size of a largest matching in a $k$-uniform hypergraph for $k \geq 3$.  Indeed, these problems were all among Karp's~\cite{K72} original twenty-one NP-complete problems.  It is also NP-complete to compute the chromatic index $\chi'(G)$ of a graph $G$~\cite{H81}.  However, every graph $G$ trivially satisfies $\chi'(G) \geq \Delta(G)$, where $\Delta(G) \coloneqq \max_{v\in V(G)} d_G(v)$ and $d_G(v) \coloneqq |\{e \in E(G) : e \ni v\}|$, and Vizing's theorem~\cite{vizing1965} implies that $\chi'(G) \leq \Delta(G) + 1$ ($\Delta(G)$ is called the \textit{maximum degree of $G$}, $d_G(v)$ is the \textit{degree of $v$ in $G$}, and these definitions, as well as the lower bound, extend naturally to hypergraphs).  More generally, it is natural to seek similar bounds for hypergraphs.

Consequently, there is a rich literature and active research on proving bounds on the chromatic index and chromatic number of hypergraphs.  As we will describe in this survey, the `R\"odl nibble' method has 
played a major role in the growth of this field.
Though this survey is mainly concerned with hypergraphs (rather than graphs), several results and colouring problems for hypergraphs arise naturally from the graph case, so we also provide the relevant context on the latter.  Similarly, we provide background on the study of matchings and independent sets in graphs and hypergraphs. 
Some of the earlier developments in the area are described in the surveys of F\"uredi~\cite{furedi1988} and Kahn~\cite{kahn1995asymptotics, kahn1997}.  Some aspects are also covered in the book of Molloy and Reed~\cite{MR02} on graph colouring with the probabilistic method.
For some recent surveys on perfect matchings in hypergraphs, see~\cite{K19, kuhn_osthus_2009, RR10, Z16}. 

\subsection{The R\"odl nibble}\label{intro:nibble}

In its basic form, the R\" odl nibble is a probabilistic approach for constructing a combinatorial substructure, such as a matching or independent set, within some host structure (such as a hypergraph) which exhibits some weak quasirandom properties.  The substructure is built bit by bit by iterating a step called a `nibble', in which elements of the host structure are selected randomly.  
This approach enabled R\"odl~\cite{rodl1985} to prove the conjecture of Erd\H os and Hanani~\cite{ErdosHanani} on the existence of approximate combinatorial designs (see Theorem~\ref{thm:erdos-hanani}) in 1985.  

Preceding R\"odl's~\cite{rodl1985} work, Ajtai, Koml{\'{o}}s, and Szemer\'edi~\cite{ajtai1981dense} showed in 1981 that a similar approach produces large independent sets in graphs with bounded average degree and no \textit{triangles} (i.e.~three pairwise adjacent vertices).  
Ultimately, the work of R\"odl~\cite{rodl1985} and of Ajtai, Koml{\'{o}}s, and Szemer\'edi~\cite{ajtai1981dense} led to numerous important developments in the theory of hypergraph colouring.  Frankl and R\"odl~\cite{FranklRodl} and Pippenger (unpublished) showed that R\"odl's `nibble' method produces nearly perfect matchings in hypergraphs under much more general conditions than those considered in~\cite{rodl1985}.  In particular, every regular uniform hypergraph with comparatively small codegree has a nearly perfect matching; Pippenger and Spencer~\cite{PippengerSpencer}, coining the term `nibble', generalized this further in 1989 by showing that $D$-regular hypergraphs have chromatic index tending to $D$ as $D\rightarrow\infty$, as long as the codegree is $o(D)$.  (Here a hypergraph $\cH$ is \textit{$D$-regular} if all of its vertices have degree $D$ and \textit{regular} if it is $D$-regular for some $D$, and the \textit{codegree} of $\cH$ is the maximum of the codegrees of all the pairs of distinct vertices of $\cH$, where the codegree of distinct vertices $u, v\in V(\cH)$ is $|\{e \in \cH : e \supseteq \{u, v\}\}|$.)  The Pippenger-Spencer theorem was further generalized to list edge-colourings by Kahn~\cite{kahn1996asymptotically} in 1996.  Meanwhile, the Ajtai-Koml{\'{o}}s-Szemer\'edi theorem~\cite{ajtai1981dense} was generalized in 1982 by Ajtai, Koml{\'o}s, Pintz, Spencer, and Szemer{\'e}di~\cite{AKPSS82}, who showed that the bound also holds for uniform hypergraphs, and also by Johansson~\cite{johansson1996} in 1996, who proved a bound on the chromatic number of triangle-free graphs of bounded maximum degree.  In 2013, Frieze and Mubayi~\cite{FM13} showed that both of these results have a common generalization in the setting of vertex-colouring hypergraphs.

These two threads of research, of edge-colouring and of vertex-colouring with the `nibble' method, have developed somewhat in parallel, sometimes intertwining.  They also both converge in the authors'~\cite{KKKMO2021} recent resolution of the Erd\H os-Faber-Lov\'asz conjecture.
Indeed, in~\cite{KKKMO2021}, we apply generalizations of the Pippenger-Spencer theorem~\cite{PippengerSpencer} as well as results inspired by Johansson's theorem~\cite{johansson1996} on vertex-colouring `locally  sparse' graphs.

\subsection{Organization of the paper}

This paper is organized as follows.  In Section~\ref{section:edge-colouring}, we survey results on hypergraph matchings and edge-colouring hypergraphs, and in Section~\ref{section:vertex-colouring}, we survey results on independent sets and vertex-colourings of graphs and hypergraphs.  In Section~\ref{section:EFL}, we present the history of the Erd\H os-Faber-Lov\'asz conjecture, and in Section~\ref{section:EFL-proof} we describe ideas involved in its recent proof~\cite{KKKMO2021}.

\subsection{Basic definitions and notation}

We say a vertex $v \in V(\cH)$ is \textit{covered} by a matching $M$ if there is an edge $e\in M$ such that $e\ni v$, and we say a set $X \subseteq V(\cH)$ is \textit{covered} by $M$ if every vertex in $X$ is covered by $M$.  A matching $M$ in $\cH$ is \textit{perfect} if it covers $V(\cH)$.  The maximum size of a matching in $\cH$, denoted $\nu(\cH)$, is called the \textit{matching number} of $\cH$.  Note that in a proper edge-colouring, each colour is assigned to the edges of a matching, and in a proper vertex-colouring, each colour is assigned to the vertices of an independent set.  

We usually denote a graph by $G$, with vertex set $V(G)$ and edge set $E(G)$.
The \textit{line graph} of a hypergraph $\cH$, denoted $L(\cH)$, is the graph $G\coloneqq L(\cH)$ where $V(G)$ is the edge set of $\cH$, and $e, f \in V(G)$ are adjacent in $G$ if $e\cap f \neq \varnothing$.  For an edge $e\in \cH$, we write $N_\cH(e)$ for short instead of $N_{L(\cH)}(e)$ to denote the neighbourhood of $e$ in the line graph of $\cH$.  Note that the matchings in $\cH$ are in one-to-one correspondence with the independent sets of $L(\cH)$ and that $\chi'(\cH) = \chi(L(\cH))$.

The \textit{fractional chromatic number} of a hypergraph $\cH$, denoted $\chi_f(\cH)$, is the smallest $k \in \mathbb R$ for which there exists a probability distribution on the independent sets of $\cH$ satisfying $\Prob{v \in I} \geq 1 / k$ for every $v\in V(\cH)$ if $I$ is drawn according to the distribution, and the \textit{fractional chromatic index} of $\cH$ is defined as $\chi'_f(\cH) = \chi_f(L(\cH))$.  The \textit{list chromatic number} of a hypergraph $\cH$, denoted $\chi_\ell(\cH)$, is the minimum $k \in \mathbb N$ such that the following holds: if $C$ is an assignment of `lists of colours' $C(v) \subseteq \mathbb N$ for each $v\in V(\cH)$ satisfying $|C(v)| \geq k$ for all $v\in V(\cH)$, then $\cH$ has a proper vertex-colouring $\phi$ such that $\phi(v) \in C(v)$ for every $v\in V(\cH)$.  The \textit{list chromatic index} of $\cH$ is defined as $\chi'_\ell(\cH) = \chi_\ell(L(\cH))$.  It is well-known that every hypergraph $\cH$ satisfies $|V(\cH)| / \alpha(\cH) \leq \chi_f(\cH) \leq \chi(\cH) \leq \chi_\ell(\cH)$ and $|\cH| / \nu(\cH) \leq \chi'_f(\cH) \leq \chi'(\cH) \leq \chi'_\ell(\cH)$.

Some authors define a hypergraph to be a pair $\cH = (V, E)$ where $E$ is a multi-set of subsets of $V$; in this survey, we refer to such an object as a \textit{multi-hypergraph}, and if every $e \in \cH$ has size two, then $\cH$ is a \textit{multigraph}.  

For $n \in \mathbb{N}$, we write $[n] := \{k \in \mathbb{N} \: : \:1 \leq k \leq n \}$. We write $c = a \pm b$ if $a-b \leq c \leq a+b$. In Sections~\ref{section:EFL} and~\ref{section:EFL-proof}, we use the `$\ll$' notation in proofs. Whenever we write a hierarchy of constants, they have to be chosen from right to left. More precisely, if we claim that a result holds whenever $0 < a \ll b \le 1$, then this means that there exists a non-decreasing function $f : (0,1] \mapsto (0,1]$ such that the result holds for all $0 < a,b \le 1$ with $a \le f(b)$. We will not calculate these functions explicitly. Hierarchies with more constants are defined in a similar way.
We use `$\log$' to denote the natural logarithm, which is relevant in Section~\ref{section:vertex-colouring}.

Our graph theory notation is standard, but one may refer to~\cite[Section 3]{KKKMO2021} for a comprehensive list of the notation we use.

\section{Matchings and edge-colouring}\label{section:edge-colouring}

\subsection{Early results}\label{edge-colouring:early}

A \textit{partial Steiner system} with parameters $(t, k, n)$ is a $k$-uniform $n$-vertex hypergraph such that every set of $t$ vertices is contained in at most one edge, and a \textit{(full) Steiner system} with parameters $(t, k, n)$ is a $k$-uniform $n$-vertex hypergraph such that every set of $t$ vertices is contained in precisely one edge.  Note that a Steiner system with parameters $(t, k, n)$ has $\left.\binom{n}{t}\big/\binom{k}{t}\right.$ edges, 
which implies that $\binom{k}{t}{\big|}\binom{n}{t}$.
The so-called Existence conjecture for designs asserts that, apart from finitely many exceptions, this and a few other trivial divisibility conditions are sufficient to ensure the existence of a Steiner system with parameters $(t, k, n)$.  In 1963, Erd\H os and Hanani~\cite{ErdosHanani} asked for an approximate version of this conjecture, which was confirmed by R\"odl~\cite{rodl1985} in 1985, initiating the use of the celebrated `nibble' method, as follows.
\begin{theorem}[R\"odl~\cite{rodl1985}]\label{thm:erdos-hanani}
For every $k > t \geq 1$ and $\eps > 0$, there exists $n_0$ such that the following holds.  For every $n \geq n_0$, there exists a partial Steiner system with parameters $(t, k, n)$ and at least $(1 - \eps)\left.\binom{n}{t}\big/\binom{k}{t}\right.$ edges.
\end{theorem}

The Existence conjecture was proved by Keevash~\cite{keevash2014existence} in 2014, by combining a generalization of Theorem~\ref{thm:erdos-hanani} (briefly discussed in Section~\ref{edge-colouring:pseudo}), with an `absorption' technique called `randomized algebraic construction'.  Glock, K\"uhn, Lo, and Osthus~\cite{GKLO16} provided a purely combinatorial proof, using `iterative absorption' instead of the algebraic approach of Keevash.

A partial Steiner system $\cH$ with parameters $(t, k, n)$ corresponds to a matching $M \coloneqq \{\binom{e}{t} : e \in \cH\}$ in the $\binom{k}{t}$-uniform auxiliary hypergraph with vertex set $\binom{V(\cH)}{t}$ and edge set $\{\binom{X}{t} : X \in \binom{V(\cH)}{k}\}$.  In particular, a Steiner system with parameters $(t, k, n)$ exists if and only if the hypergraph $\cH^*_{t,k,n}$ with vertex set $\binom{[n]}{t}$ and edge set $\{\binom{X}{t} : X \in \binom{[n]}{k}\}$ has a perfect matching, and Theorem~\ref{thm:erdos-hanani} is equivalent to the statement that $\cH^*_{t,k,n}$ contains a matching covering all but a vanishing proportion of its vertices as $n\rightarrow \infty$.  
This result holds much more generally for hypergraphs satisfying mild pseudorandomness conditions involving the degrees and codegrees.
Indeed, Frankl and R\"odl~\cite{FranklRodl} proved that if $\cH$ is an $N$-vertex, $D$-regular hypergraph with codegree at most $D / (\log N)^4$, then $\cH$ has a matching covering all but $o(N)$ vertices as $N\rightarrow\infty$.  Since $\cH^*_{t,k,n}$ is $\binom{n - t}{k - t}$-regular and has codegree at most $\binom{n - t - 1}{k - t - 1}$, this result generalizes Theorem~\ref{thm:erdos-hanani}.
Pippenger generalized this result further, by relaxing the codegree condition, as follows.  (Pippenger's result was not published, but a proof is given e.g.~in~\cite[Theorem 8.4]{furedi1988}.)

\begin{theorem}[Pippenger]\label{thm:pippenger}
    For every $k, \eps > 0$, there exists $\delta > 0$ such that the following holds.
    If $\cH$ is an $n$-vertex $k$-uniform $D$-regular hypergraph with codegree at most $\delta D$, then there is a matching in $\cH$ covering all but at most $\eps n$ vertices.
\end{theorem}

As mentioned in Section~\ref{intro:nibble}, Theorem~\ref{thm:pippenger} is proved with the nibble method, which we now sketch in more detail.  Each step of the nibble process produces a matching in a nearly $D_i$-regular subhypergraph $\cH_i\subseteq\cH$ (beginning with $\cH_1 \coloneqq \cH$ and $D_1 \coloneqq D$), as follows.  First, select a set of edges $X_i\subseteq \cH_i$ randomly, where each edge $e\in\cH_i$ is included in $X_i$ independently with probability $\eps' / D_i$ where $\eps' > 0$ is a small constant depending on $k$ and $\eps$.  Then, let $N_i\subseteq X_i$ be the matching consisting of the edges of $X_i$ that are disjoint from the rest\COMMENT{i.e.~$N_i\coloneqq\{e \in X_i : e\cap f = \varnothing~\forall f \in X_i\setminus\{e\}\}$}.  Crucially, each vertex is in an edge of $X_i$ with probability close to $1 - e^{-\eps'}$, and moreover, the codegree condition ensures that two distinct vertices are in an edge of $X_i$ somewhat independently.  This fact implies that the hypergraph $\cH_{i+1} $ obtained by removing every vertex in an edge of $X_i$ is nearly $D_{i+1}$-regular, where $D_{i+1} \coloneqq e^{-\eps'(k - 1)}D_i$, which in turn allows the nibble process to continue.  
Each edge $e \in \cH$ is in $X_i$ with probability roughly $(\eps' / D_i)e^{-\eps' k(i - 1)} =\COMMENT{$(\eps' / D)e^{\eps'(k - 1)(i - 1)}e^{-\eps' k(i - 1)} =$} (\eps' / D) e^{-\eps'(i - 1)}$ (indeed, $e$ is in $\cH_i$ with probability roughly $e^{-\eps' k (i - 1)}$ and conditioning on this, is selected in $X_i$ with probability $\eps' / D_i$).  
Each edge in $X_i$ is then kept in $N_i$ with probability roughly $e^{-\eps' k}$.  Thus, after $t$ steps of the nibble, each edge $e\in\cH$ is contained in $M\coloneqq \bigcup_{i=1}^t N_i$ with probability close to $\sum_{i=1}^t(\eps' / D)e^{-\eps'(i - 1)}e^{-\eps' k} =\COMMENT{$(\eps' / D) e^{-\eps' k}\sum_{i=1}^t e^{-\eps'(i - 1)}$} \alpha / D$, where $\alpha \coloneqq \eps'e^{-\eps' k}(1 - e^{-\eps' t}) / (1 - e^{-\eps'})$.  In particular, $M$ is a matching and the expected number of uncovered vertices is essentially at most $(1 - \alpha)n \leq \eps n$ (if $\eps'$ and $t^{-1}$ are small enough)\COMMENT{by the linearity of expectation, and since $\cH$ is $D$-regular, the expected size of $M$ is roughly $(\alpha / D)|\cH| = \alpha n / k$, and since $\cH$ is $k$-uniform, $M$ covers roughly $\alpha n$ vertices in expectation}.  Kahn~\cite{kahn1996linear} and Kahn and Kayll~\cite{KK97} proved generalizations of Theorem~\ref{thm:pippenger} where the regularity and codegree conditions are replaced with the existence of a fractional matching satisfying a certain `local sparsity' condition, which can be used to guide the nibble process. 

It is natural to wonder if the `random greedy algorithm' (which would select $X_i$ to consist of a single edge chosen uniformly at random from $\cH_i$ in the process above) also produces a nearly perfect matching under the conditions of Theorem~\ref{thm:pippenger}.  Indeed, this result was obtained independently by Spencer~\cite{spencer1995asymptotic} and by R\"odl and Thoma~\cite{rodl1996asymptotic}.  To prove this, Spencer~\cite{spencer1995asymptotic} considered a branching process, and R\"{o}dl and Thoma~\cite{rodl1996asymptotic} showed that the random greedy algorithm produces a matching with a similar distribution as the nibble process.  Note that these results immediately yield a (randomized) polynomial-time Monte Carlo algorithm for finding the matching guaranteed by Theorem~\ref{thm:pippenger}.  The proof of Theorem~\ref{thm:pippenger} given in~\cite{PippengerSpencer} also yields such an algorithm.
R\"{o}dl and Thoma\COMMENT{at RSA in 1995 (see~\cite{grable1996nearly})} also asked if there is an NC-algorithm (and in particular, a deterministic, polynomial-time algorithm) for finding such a matching, and  Grable~\cite{grable1996nearly} answered their question in the affirmative.  

In 1989, Pippenger and Spencer~\cite{PippengerSpencer} generalized Theorem~\ref{thm:pippenger} to edge-colouring, as follows.
\begin{theorem}[Pippenger and Spencer~\cite{PippengerSpencer}]\label{thm:pippenger-spencer}
For every $k, \eps > 0$, there exists $\delta > 0$ such that the following holds.\COMMENT{Pippenger and Spencer also have an $n_0$, but it is redundant.  The codegree condition implies that $n \geq D \geq 1/\delta$.}
If $\cH$ is a $k$-uniform $D$-regular hypergraph of codegree at most $\delta D$, then $\chi'(\cH) \leq (1 + \eps)D$.
\end{theorem}

Since every hypergraph $\cH$ satisfies $\nu(\cH) \geq |\cH| / \chi'(\cH)$ and moreover $|\cH| = D |V(\cH)| / k$ if $\cH$ is $D$-regular and $k$-uniform, Theorem~\ref{thm:pippenger-spencer} implies Theorem~\ref{thm:pippenger}. In fact, in Pippenger and Spencer's~\cite{PippengerSpencer} proof of Theorem~\ref{thm:pippenger-spencer}, they used the argument described above to select nearly perfect matchings randomly with the nibble process, which ultimately form most of the colour classes.  Roughly, they show that after selecting $D$ such matchings---in groups of size $o(D)$, selected iteratively in a `semi-random' way (which could also be considered a nibble process)---the remaining hypergraph has small maximum degree and can thus be properly edge-coloured with at most $\eps D$ colours in a `greedy' fashion.

Pippenger and Spencer~\cite{PippengerSpencer} actually proved the slightly stronger version of Theorem~\ref{thm:pippenger-spencer} that applies if every vertex of $\cH$ has degree $(1 \pm \delta)D$, rather than precisely $D$.  Kahn~\cite{kahn1992coloring} observed that Theorem~\ref{thm:pippenger-spencer} holds more generally for $k$-bounded hypergraphs of maximum degree at most $D$, by showing that such hypergraphs can be `embedded' in a nearly $D$-regular $k$-uniform hypergraph with the same or larger chromatic index (a hypergraph $\cH$ is \textit{$k$-bounded} if every $e\in\cH$ satisfies $|e| \leq k$).  This sequence of results culminated in Kahn's~\cite{kahn1996asymptotically} generalization of the Pippenger-Spencer theorem to list colouring, as follows.

\begin{theorem}[Kahn~\cite{kahn1996asymptotically}]\label{thm:asym-list-edge-col}
For every $k, \eps > 0$, there exists $\delta > 0$ such that the following holds.  If $\cH$ is a $k$-bounded hypergraph of maximum degree at most $D$ and codegree at most $\delta D$, then $\chi'_\ell(\cH) \leq (1 + \eps)D$.
\end{theorem}

The so-called `List Edge Colouring conjecture'---first posed by Vizing in 1975 and asked by many others since (see e.g.~\cite{JT95})---asserts that every graph $G$ satisfies $\chi'_\ell(G) = \chi'(G)$, and Theorem~\ref{thm:asym-list-edge-col} for $k = 2$ confirms this conjecture asymptotically.  Kahn's proof of Theorem~\ref{thm:asym-list-edge-col} is also based on a nibble argument but is notably different from Pippenger and Spencer's~\cite{PippengerSpencer} proof of Theorem~\ref{thm:pippenger-spencer}.  In particular, rather than selecting colour classes one by one, in each step of the nibble, edges are assigned a colour randomly from their lists, so the colour classes are constructed in parallel.

For \textit{linear} hypergraphs (i.e., hypergraphs of codegree one), Molloy~\cite[Theorem 10]{M18} generalized Theorem~\ref{thm:asym-list-edge-col} to the setting of `correspondence colouring' (also known as DP-colouring) in 2018, and Bonamy, Delcourt, Lang, and Postle~\cite[Theorem 7]{BDLP20} generalized this result further by proving a `local version'. 

Several results also strengthen Theorems~\ref{thm:pippenger}--\ref{thm:asym-list-edge-col} by improving the asymptotic error terms.  This is the focus of Section~\ref{edge-colouring:asymptotic}. We conclude this subsection by discussing two open problems from the late 1990s.
Both of these are conjectured to hold for multi-hypergraphs.  In fact, Theorems~\ref{thm:pippenger}--\ref{thm:asym-list-edge-col} also hold for multi-hypergraphs (but the codegree conditions also bound the number of copies of each edge).

\begin{conjecture}[Kahn~\cite{kahn1996asymptotically}]\label{conj:kahn-frac}
  For every $k, \eps > 0$, there exists $K$ such that the following holds.  If $\cH$ is a $k$-bounded multi-hypergraph, then $\chi'_\ell(\cH) \leq \max\{(1 + \eps)\chi'_f(\cH), K\}$.
\end{conjecture}

Even the weaker version of this conjecture, with the list chromatic index replaced by the chromatic index, is wide open.  Only the case $k = 2$ is known.  For 2-bounded hypergraphs (i.e.~graphs with edge-multiplicity $1$), the conjecture follows from Vizing's theorem~\cite{vizing1965} for the chromatic index and from Theorem~\ref{thm:asym-list-edge-col} for the list chromatic index\COMMENT{since every graph $G$ satisfies $\chi'_f(G) \geq \Delta(G)$. (If $v\in V(G)$ has maximum degree in $G$, then if a matching $M$ in $G$ is randomly sampled according to a probability distribution satisfying $\Prob{e \in M} \geq 1 / k$ for each $e \in M$, then $v$ is covered by an edge of $M$ with probability at least $\Delta(G) / k$.  In particular, $\Delta(G) \leq k$.)}.  
As shown by Seymour~\cite{seymour1979} using Edmonds' Matching Polytope theorem~\cite{edmonds1965}, every multigraph $G$ satisfies $\chi'_f(G) = \max\{\Delta(G), \Gamma(G)\}$, where 
\begin{equation*}
    \Gamma(G) \coloneqq \max\left\{ \frac{2|E(H)|}{|V(H)|-1} : {H \subseteq G},~|V(H)|\geq 3~\text{and is odd}\right\}.
\end{equation*}
Kahn~\cite{kahn1996multi} proved that every multigraph $G$ satisfies $\chi'(G) \leq (1 + o(1))\chi_f(G)$ and in~\cite{K00} extended this result to list colouring, thus confirming Conjecture~\ref{conj:kahn-frac} in full for the case $k = 2$.  For the ordinary chromatic index, even more is now known in this case.  In the 1970s, Goldberg~\cite{goldberg1973} and Seymour~\cite{seymour1979} independently conjectured\COMMENT{Andersen~\cite{andersen1977} also makes a conjecture that he showed is equivalent to the Goldberg-Seymour conjecture, citing Goldberg~\cite{goldberg1973}.} that every multigraph $G$ satisfies $\chi'(G) \leq \max\{\Delta(G) + 1 , \lceil\Gamma(G)\rceil\}$.  Thus, Kahn's result~\cite{kahn1996multi} confirmed the Goldberg-Seymour conjecture asymptotically.
Recently, a full proof (which does not rely on probabilistic arguments) of the Goldberg-Seymour conjecture was obtained by Chen, Jing, and Zang~\cite{cjz2019}.

The next conjecture was posed by Alon and Kim~\cite{AK97}.  A hypergraph $\cH$ is called \textit{$t$-simple} if every two distinct edges of $\cH$ have at most $t$ vertices in common; in particular, a hypergraph is $1$-simple if and only if it is linear.

\begin{conjecture}[Alon and Kim~\cite{AK97}]\label{conj:alon-kim}
 For  every $k\geq t \geq 1$ and $\eps > 0$, there exists $D_0$ such that the following holds.  For every $D \geq D_0$, if $\cH$ is a $k$-uniform, $t$-simple multi-hypergraph with maximum degree at most $D$, then
 \begin{equation*}
 \chi'(\cH) \leq (t - 1 + 1/t + \eps)D.
 \end{equation*}
\end{conjecture}

The conjecture is true for $k = 2$ by Vizing's theorem~\cite{vizing1965} for $t = 1$  and by a result of Shannon~\cite{S49} for $t = 2$.  For $k > t = 1$, the conjecture follows from Theorem~\ref{thm:pippenger-spencer} (together with the observation of Kahn in~\cite{kahn1996asymptotically} mentioned above).  Kahn (see~\cite{AK97}) conjectured that the $t$-simple condition in Conjecture~\ref{conj:alon-kim} can be relaxed to requiring that the `$(t+1)$-codegrees' are small (i.e.~ every set of $t + 1$ vertices is contained in at most $\delta D$ edges, for some $\delta > 0$), which if true, would generalize Theorem~\ref{thm:pippenger-spencer}.  The remaining cases are still open. The case $k = t$ (without the `$ + \eps$' in the bound) was proved by F{\"u}redi, Kahn, and Seymour~\cite{FKS93} for the fractional chromatic index.  

Alon and Kim~\cite{AK97} showed that Conjecture~\ref{conj:alon-kim} holds for \textit{intersecting} hypergraphs (i.e.~hypergraphs with matching number one), and they gave a construction to show that if Conjecture \ref{conj:alon-kim} is true, then it would be asymptotically tight for every $k \ge t$ for which there exists a projective plane of order $t-1$. We sketch their construction here. Let $D$ be a large integer divisible by $t$, let $m \coloneqq t^2 - t + 1$, and fix a projective plane $P$ of order $t-1$ with $m$ lines $\ell_1, \ell_2, \dots, \ell_m$ on a set of $m$ points. For each of the lines $\ell_i$, let $\mathcal F_i$ be a collection of $D/t$ sets of size $k$ containing $\ell_i$, so that all the $m D/t$ sets $\{A \setminus \ell_i : 1 \leq i \leq m \text{ and } A \in \mathcal F_i\}$ are pairwise disjoint and disjoint from $P$. Let $\cH$ be the $k$-uniform hypergraph whose edge-set is $\bigcup_i \mathcal F_i$. Then clearly $\cH$ is intersecting, $k$-uniform, and $t$-simple, its maximum degree is at most $D$, and it has $mD/t = (t -1 + 1/t)D$ edges. Thus, $\chi'(\cH) \ge (t -1 + 1/t)D$.

\subsection{Asymptotic improvements}\label{edge-colouring:asymptotic}

Let $\cH$ be a $k$-uniform, $D$-regular hypergraph on $n$ vertices. Recall that Pippenger's theorem (Theorem~\ref{thm:pippenger}) shows that if the codegree of $\cH$ is $o(D)$, then there is a matching in $\cH$ covering all but at most $o(n)$ vertices. However, his proof does not supply an explicit estimate for the error term $o(n)$. Also recall that Theorems~\ref{thm:pippenger-spencer} and~\ref{thm:asym-list-edge-col} imply that if the codegree of $\cH$ is $o(D)$, then the chromatic index of $\cH$ is $D + o(D)$. Sharpening these error terms is useful for many applications, and considerable progress has been made towards this end with improved analysis and variations of the nibble method, with more powerful concentration inequalities. In this subsection we will discuss many such results.

Grable~\cite{grable1999more} proved that if the codegree is at most $D^{1 - \delta}$ in Theorem~\ref{thm:pippenger} then there is a matching covering all but at most $n (D / \log n)^{- \delta / (2k - 1 + o(1))}$ vertices. In 1997, Alon, Kim, and Spencer~\cite{AKS1997} improved this bound for linear hypergraphs by showing the following.

\begin{theorem}[Alon, Kim, and Spencer~\cite{AKS1997}]\label{thm:aks}
  Let $k \geq 3$. Let $\cH$ be a $k$-uniform $D$-regular $n$-vertex  linear hypergraph. Then $\cH$ has a matching containing all but at most $O(nD^{-\frac{1}{k-1}} \log^{c_k} D)$ vertices, where $c_k = 0$ for $k > 3$ and $c_3 = 3/2$.
\end{theorem}

Based on computer simulations (see e.g.~\cite{alonspencer2016}), Alon, Kim, and Spencer conjectured that the simple random greedy algorithm outlined in the previous subsection should also produce a matching containing all but at most $O(nD^{-\frac{1}{k-1}} \log^{O(1)} D)$ vertices.  The results of Spencer~\cite{spencer1995asymptotic} and of R\"odl and Thoma~\cite{rodl1996asymptotic} mentioned after Theorem~\ref{thm:pippenger} only show that the random greedy algorithm produces a matching covering all but $o(n)$ vertices.  This error term was sharpened by Wormald~\cite{wormald1999} and Bennett and Bohman~\cite{BB2019} but the conjecture is still open.

Kostochka and R\"odl~\cite{KR1998} extended Theorem~\ref{thm:aks} to hypergraphs with small codegrees $C$ (i.e., satisfying $C \leq D^{1-\gamma}$ for some $\gamma > 0$). In 2000, Vu~\cite{vu2000} further extended the result of Kostochka and R\"odl~\cite{KR1998} by removing the assumption $C < D^{1-\gamma}$ on the codegree. More precisely, he showed that every $k$-uniform $D$-regular $n$-vertex hypergraph with codegree at most $C$ contains a matching covering all but at most $O(n({D}/{C})^{-\frac{1}{k-1}}\log^c D)$ vertices for some constant $c > 0$.  He also obtained stronger bounds if one makes additional assumptions on the `$t$-codegrees' for $t > 2$.

Very recently, Kang, K\"{u}hn, Methuku, and Osthus~\cite{KKMO2020} improved Theorem~\ref{thm:aks}, and the results of Kostochka and R\"odl~\cite{KR1998}, and of Vu~\cite{vu2000} for hypergraphs with small codegree. In the case when $\cH$ is linear, they showed the following.

\begin{theorem}[Kang, K\"{u}hn, Methuku, and Osthus~\cite{KKMO2020}] \label{thm:kkmo}
 Let $k > 3$ and let $0 < \gamma,\mu < 1$ and 
		$0 < \eta < \frac{k-3}{(k-1)(k^3 - 2k^2 - k + 4)}$. Then there exists $n_0 = n_0(k, \gamma , \eta ,\mu)$ such that the following holds for $n \geq n_0$ and $D \geq \exp(\log^\mu n)$.
		

	If $\cH$ is a $k$-uniform $D$-regular linear hypergraph on $n$ vertices, then $\cH$ contains a matching covering all but at most $n D^{-\frac{1}{k-1} - \eta}$ vertices.	
\end{theorem}

Their approach consists of showing that the R\"odl nibble process not only constructs a large matching but it also produces many well-distributed `augmenting stars' which can then be used to significantly augment the matching constructed by the R\"odl nibble process.

Below we discuss results concerning improvements on the chromatic index of hypergraphs. In 2000, Molloy and Reed~\cite{MR2000} sharpened the error term in Theorem~\ref{thm:asym-list-edge-col}. For linear hypergraphs their result can be stated as follows. 

\begin{theorem}[Molloy and Reed~\cite{MR2000}]\label{thm:molloy-reed}
  If $\cH$ is a $k$-uniform linear hypergraph with maximum degree at most $D$, then $\chi_{\ell}'(\cH) \leq D + O(D^{1-1/k} \log^4 D)$.
\end{theorem}

For graphs, this result improves a result of H{\"a}ggkvist and Janssen~\cite{HJ97} and provides the best known general bound for the List Edge Colouring conjecture.%
\COMMENT{(H{\"a}ggkvist and Janssen~\cite{HJ97} also proved that every $n$-vertex complete graph has list chromatic index at most $n$, and Schauz~\cite{S14} proved that if $n - 1$ is an odd prime, then every $n$-vertex complete graph has list chromatic index at most $n - 1$.  These results confirm the List Edge Colouring conjecture in their respective cases, but the conjecture is still not known to hold for all complete graphs).}  
Molloy and Reed~\cite{MR2000} actually proved a more general result showing that every $k$-uniform hypergraph $\cH$ with maximum degree at most $D$ and codegree at most $C$ has list chromatic index at most $D + O(D(D/C)^{-1/k} (\log D / C)^4)$, which also gave the best known bound on the ordinary chromatic index $\chi'(\cH)$. Very recently, Kang, K\"{u}hn, Methuku, and Osthus~\cite{KKMO2020} showed that this bound on the chromatic index can be improved further. For linear hypergraphs their result can be stated as follows.

\begin{theorem}[Kang, K\"{u}hn, Methuku, and Osthus~\cite{KKMO2020}] 
 \label{thm:kkmo-colouring}
Let $k \geq 3$, and let $D, n, C > 0$, $0 < \gamma,\mu < 1$, and $0 < \eta < \frac{k-2}{k(k^3 + k^2 - 2k + 2)}$. Then there exists $n_0 = n_0(k , \gamma , \eta, \mu)$ 
	such that the following holds for all $n \geq n_0$ and $D \geq \exp(\log^\mu n)$.
		
	If $\cH$ is a $k$-uniform linear hypergraph on $n$ vertices with maximum degree at most $D$, then $\chi'(\cH) \leq D + D^{1 - 1/k - \eta}$.
\end{theorem}

Theorems~\ref{thm:kkmo}--\ref{thm:kkmo-colouring} are unlikely to be best possible.  The best lower bounds we know come from the following construction, in which every matching leaves $\Omega(n / D)$ vertices uncovered. Consider an $m$-vertex $k$-uniform $D$-regular linear hypergraph $\cH$ such that $m = O(D)$ and $m-1$ is divisible by $k$ (e.g., using a Steiner system $S(2,k,m)$), so the union of $n/m$ disjoint copies of $\cH$ yields an $n$-vertex $k$-uniform $D$-regular hypergraph with at least $\Omega(n/D)$ vertices uncovered by any matching.

If we know more about the hypergraph $\cH$, then the bound given in Theorem~\ref{thm:kkmo} can be improved further.  For example, if $\cH$ is a Steiner triple system (i.e.,~a Steiner system with parameters $(2, 3, n)$), then $\cH$ is $(n-1)/2$-regular and linear, and Brouwer~\cite{brouwer1981size} conjectured the following in 1981.

\begin{conjecture}[Brouwer~\cite{brouwer1981size}]\label{conj:brouwer}
Every Steiner triple system with $n$ vertices has a matching of size at least $\frac{n-4}{3}$.
\end{conjecture}

Recently, combining the nibble method with the robust expansion
properties of edge-coloured pseudorandom graphs, Keevash, Pokrovskiy, Sudakov, and Yepremyan~\cite{KPSY2020} showed that every Steiner triple system has a matching covering all but at most $O(\log n / \log \log n)$ vertices. 

A related problem is a famous conjecture of Ryser, Brualdi and Stein~\cite{br1991, ryser1967, stein1975} which states that every $n \times n$ Latin square has a transversal of order $n - 1$ and moreover, if $n$ is odd, then it has a full transversal. The best known bound for this problem was given in~\cite{KPSY2020} where the authors showed that every $n \times n$ Latin square contains a transversal of order $n - O({\log n}/{\log \log n})$. The problem of finding large transversals in Latin squares can be rephrased as a problem about finding large matchings in hypergraphs. Indeed, we can construct a $3$-uniform hypergraph $\cH_{\cL}$ on $3n$ vertices from an $n \times n$ Latin square $\cL$ as follows. The vertex set of $\cH_{\cL}$ is $V(\cH_{\cL}) = R \cup C \cup S$ where $R$, $C$ and $S$ are the rows, columns and symbols of $\cL$. For every entry of $\cL$ we add an edge to $\cH_{\cL}$ -- if the $(i,j)$-th entry of $\cL$ contains a symbol $s$ then we add the edge $\{i,j,s\}$ to $\cH_{\cL}$. Clearly, $\cH_{\cL}$ is $n$-regular and $3$-partite, and a matching in $\cH$ corresponds to a transversal in $\cL$. Thus the Ryser-Brualdi-Stein conjecture can be regarded as the `partite-version' of Brouwer's conjecture.

Similarly, it is interesting to determine the maximum chromatic index of an $n$-vertex Steiner triple system (or an $n \times n$ Latin square).  Meszka, Nedela, and Rosa~\cite{MNR2006} conjectured the following in 2006.

\begin{conjecture}[Meszka, Nedela, and Rosa~\cite{MNR2006}]\label{conj:mnr}
If $\cH$ is a Steiner triple system with $n > 7$ vertices, then $\chi'(\cH) \leq (n - 1)/2 + 3$ and moreover, if $n \equiv 3~(\mathrm{mod}~6)$, then $\chi'(\cH) \leq (n - 1)/2 + 2$.
\end{conjecture}

Since an $n$-vertex Steiner triple system is $(n - 1)/2$-regular, it is obvious that $\chi'(\cH)\geq(n - 1)/2$, and equality holds if and only if $\cH$ can be decomposed into perfect matchings. Hence, if $n \equiv 1~(\mathrm{mod}~6)$, then $\chi'(\cH)\geq(n + 1)/2$.
In fact, there are constructions of Steiner triple systems with $n$ vertices which show that Conjecture~\ref{conj:mnr}, if true, is tight ~\cite{BCHW2017, meszka2013, RW1971, VSSRRCCC1993}. 
Similarly, for Latin squares the following conjecture was posed independently by Cavenagh and Kuhl~\cite{ck2015} in 2015 and Besharati,  Goddyn,  Mahmoodian, and Mortezaeefarbeen~\cite{bgmm2016} in 2016.

\begin{conjecture}\label{conj:latin}
    Let $\cL$ be an $n \times n$ Latin square.  If $\cH_{\cL}$ is the corresponding $3$-uniform $3$-partite hypergraph, then $\chi'(\cH_{\cL}) \leq n+2$ and moreover, if $n$ is odd, then $\chi'(\cH_{\cL}) \leq n+1$.
\end{conjecture}

Note that Conjecture~\ref{conj:mnr} implies Conjecture~\ref{conj:brouwer}\COMMENT{Suppose for a contradiction that $\cH$ is a Steiner triple system with no matching of size at least $(n - 4)/3$ and $\chi'(\cH) \leq (n + 1)/2 + 2$.  Since $\nu(\cH)\chi'(\cH) \geq |\cH| = n(n - 1)/6$ and $\nu(\cH) \leq (n - 6)/3$ (as $\nu(\cH)$ is an integer and $n$ is congruent to $1$ or $3$ mod $6$), we have $(n^2 - n - 30)/6 = ((n - 6)/3)((n + 5)/2) \geq (n^2 - n)/6$, a contradiction.}, and Conjecture~\ref{conj:latin} implies the Ryser-Brualdi-Stein conjecture\COMMENT{Suppose for a contradiction that $\cL$ is a Latin square with no transversal of size $n - 1$ such that $\chi'(\cH_\cL) \leq n + 2$.  Since $\nu(\cH_\cL)\chi'(\cH_\cL) \geq |\cH_L| = n^2$ and $\nu(\cH_\cL) \leq n - 2$, we have $n^2 - 4 = (n - 2)(n + 2) \geq n^2$, a contradiction.  Similarly, if $n$ is odd, $\cL$ has no transversal of size $n$, and $\chi'(\cH_\cL) \leq n + 1$, then we have $(n - 1)(n + 1) \geq n^2$, a contradiction. }. 
Theorem~\ref{thm:kkmo} implies that every $n$-vertex Steiner triple system has chromatic index at most $n/2 + O(n^{2/3 - 1/100})$ and every hypergraph corresponding to an $n \times n$ Latin square has chromatic index at most $n + O(n^{2/3 - 1/100})$; currently these bounds are the best known.

\subsection{Pseudorandom hypergraph matchings}\label{edge-colouring:pseudo}

Let $\cH$ be a $k$-uniform $D$-regular hypergraph on $n$ vertices, and let $M\subseteq\cH$ be a random matching generated by the nibble process, such that $M$ covers all but at most $\eps n$ vertices of $\cH$ (with high probability), where $\eps \in (0, 1)$.  A heuristic argument suggests that each vertex of $\cH$ is left uncovered by $M$ roughly independently with probability $\eps$.  In many applications (including our proof in~\cite{KKKMO2021}), it is useful to find a nearly perfect matching guaranteed by Theorem~\ref{thm:pippenger} with additional `pseudorandom' properties that are compatible with this heuristic.  
In this subsection, we discuss some results that provide nearly perfect pseudorandom hypergraph matchings and some of their applications.
In particular, we show how a `pseudorandom version' of Pippenger's theorem (Theorem~\ref{thm:pippenger}) is in fact equivalent to the Pippenger-Spencer theorem (Theorem~\ref{thm:pippenger-spencer}). 

The first pseudorandom hypergraph matching result of this sort was proved by Alon and Yuster~\cite{AY2005} in 2005.\COMMENT{The following is a weaker but more concise, and still applicable version of the Alon-Yuster theorem~\cite{AY2005}:
\begin{theorem}[Alon and Yuster~\cite{AY2005}]\label{thm:alon-yuster}
  For every $k, \eps > 0$ there exist $\delta, n_0 > 0$ such that the following holds.  Suppose $\cH$ is an $n$-vertex $k$-uniform hypergraph and $\cF$ is a collection of subsets of $V(\cH)$ such that $|\cF| \leq n^{\log n}$.  If $n \geq n_0$, every vertex of $\cH$ has degree $(1 \pm \delta)D$, and $\cH$ has codegree at most $D / \log^{9k} n$, then there exists a matching $M$ of $\cH$ such that every $S \in \cF$ with $|S| \geq D^{1/3}\log^3n$ satisfies $|S \setminus V(M)| \leq \eps |S|$.
\end{theorem}}
With a slightly stronger assumption regarding codegrees,
Ehard, Glock, and Joos~\cite{EGJ2020} recently proved a stronger and more flexible version. 
The following is an immediate corollary of~\cite[Theorem 1.2]{EGJ2020}.

\begin{theorem}[Ehard, Glock, and Joos~\cite{EGJ2020}]\label{thm:egj}
    For every $k\geq 2$ and $\delta \in (0, 1)$, there exists $D_0$ such that the following holds for all $D\geq D_0$ and $\eps \coloneqq \delta / (50 k^2)$.  Suppose that $\cH$ is a $k$-uniform hypergraph and $\cF$ is a collection of subsets of $V(\cH)$ such that $|\cF| \leq \exp(D^{\eps^2})$ and $\sum_{v\in S}d_\cH(v) \geq k D^{1 + \delta}$ for every $S \in \cF$.  If $\cH$ has maximum degree at most $D$, codegree at most $D^{1 - \delta}$, and $e(\cH) \leq \exp(D^{\eps^2})$, then there exists a matching $M$ of $\cH$ such that every $S \in \cF$ satisfies $|S \cap V(M)| = (1 \pm D^{-\eps})\sum_{v \in S}d_\cH(v) / D$.
\end{theorem}

Note that if $\cH$ is $D$-regular in Theorem~\ref{thm:egj} and $V(\cH) \in \cF$, then $M$ covers all but at most $nD^{-\eps}$ vertices of $\cH$.  Moreover, for every $S\in\cF$, at most $|S|D^{-\eps}$ vertices are uncovered by $M$, as we would expect if each vertex was uncovered with probability $D^{-\eps}$.
Ehard, Glock, and Joos~\cite{EGJ2020} actually proved a stronger version of Theorem~\ref{thm:egj} involving weight functions on the edges of $\cH$ of the form $\omega : \cH \rightarrow \mathbb R_{\geq0}$.  The `pseudorandomness' heuristic suggests that every edge is in $M$ with probability $1 / D$, and thus the expected total weight of edges in $M$ should be $\sum_{e\in\cH}\omega(e) / D$.

Hypergraph matching results, particularly ones with pseudorandomness guarantees, are widely applicable in combinatorics and beyond.  We give some examples here.  Ford, Green, Konyagin, Maynard, and Tao~\cite{FGKMT17} proved a pseudorandom generalization of Theorem~\ref{thm:pippenger} (stated for \textit{coverings} rather than matchings and which allows for non-uniform hypergraphs).  They used it to improve bounds on gaps between prime numbers.  As we saw in Section~\ref{edge-colouring:asymptotic}, in~\cite{KKMO2020} and~\cite{KPSY2020}, pseudorandom properties of hypergraph matchings can be `bootstrapped' to produce a larger matching.  Furthermore, in some applications of the `absorption method', such as~\cite{CDGKO20, FJS19, GKLO16, GKMO21, keevash2014existence}, a matching in an auxiliary hypergraph is used to construct a nearly spanning structure which is complemented by an `absorbing structure', so that the pseudorandom properties can be exploited for `absorption', which results in a spanning structure\COMMENT{The `absorption' argument in the proofs of the existence of designs in~\cite{keevash2014existence} and~\cite{GKLO16} is actually used to obtain a decomposition, but as discussed in Section~\ref{edge-colouring:early}, $(t, k, n)$-Steiner systems correspond to perfect matchings in $\cH^*_{t, k, n}$.}.  
Hypergraph matchings with pseudorandomness properties can also be used to construct approximate decompositions (see e.g.~\cite{GJKKO21, GKLO16, keevash2014existence, KKKO20}) or edge-colourings.  Indeed, in the proof of the Erd\H os-Faber-Lov\'asz conjecture, we use Theorem~\ref{thm:egj} to obtain a partial edge-colouring of a linear hypergraph in which each colour class has pseudorandom properties that enable some of the uncovered vertices to be absorbed\COMMENT{The absorption argument uses Hall's theorem~\cite{H35} applied to a set of `reservoir' edges to find a matching augmenting each colour class.  The pseudorandomness of each colour class is essential for verifying that the Hall's condition holds.  The proof of Kom{l\'{o}}s, S{\'{a}}rk{\"{o}}zy, and Szemer{\'{e}}di~\cite{kss1997} of the `Blow-up lemma' could be considered a precursor to this argument.} (see Section~\ref{efl:small} for more details).  As an illustration of this approach, we show how Theorem~\ref{thm:egj} implies a version of the Pippenger-Spencer theorem.  
First we need the following definition and observation.



\begin{definition}\label{def:incidence-hypergraph}
  For every hypergraph $\cH$ and $t \in \mathbb N$, we define the \textit{$t$-wise incidence hypergraph} $\cH^* \coloneqq \hgInc_t(\cH)$ to be the hypergraph with 
  \begin{itemize}
      \item vertex set $\cH \cup \left([t]\times V(\cH)\right)$ and
      \item edge set $\{\{e\}\cup (\{i\}\times e) : e \in \cH, i \in [t]\}$.
  \end{itemize}
\end{definition}
That is, for every $e = \{v_1, \dots, v_k\} \in \cH$, we include $t$ edges in the $t$-wise incidence hypergraph $\cH^* \coloneqq \hgInc_t(\cH)$, where each such edge is of the form $\{e, (i, v_1), \dots, (i, v_k)\}$ for some $i \in [t]$.

\begin{obs}\label{obs:incidence-hypergraph}
  Let $\cH$ be a hypergraph, and let $\cH^* \coloneqq \hgInc_t(\cH)$ be the $t$-wise incidence hypergraph.  The following holds.
  \begin{enumerate}[(a)]
      \item\label{incidence-hypergraph:uniformity} If $\cH$ is $k$-uniform, then $\cH^*$ is $(k + 1)$-uniform.
      \item\label{incidence-hypergraph:codegree} The codegree of $\cH^*$ is at most the codegree of $\cH$.
      \item\label{incidence-hypergraph:degree} For every $v \in V(\cH)$ and $i \in [t]$, $d_{\cH^*}((i, v)) = d_\cH(v)$, and for every $e \in \cH$, $d_{\cH^*}(e) = t$.
      \item\label{incidence-hypergraph:colouring}  A set $M \subseteq \cH^*$ is a matching if and only if $M_1, \dots, M_t$, where $M_i \coloneqq \{e \in \cH : \exists f \in M,~f\supseteq \{i\}\times e\}$, are pairwise edge-disjoint matchings in $\cH$.  In particular, the chromatic index of $\cH$ is at most $t$ if and only if $\cH^*$ contains a matching covering $\cH$.
  \end{enumerate}
\end{obs}

Using Observation~\ref{obs:incidence-hypergraph}, we can show that, under a slightly stronger codegree condition, Theorem~\ref{thm:egj} implies Theorem~\ref{thm:pippenger-spencer} (i.e.~that the chromatic index of a $D$-regular hypergraph of small codegree tends to $D$), as follows.  
Let $k\geq 2$, $\delta\in(0, 1)$, and $\eps \coloneqq \delta/(50(k+1)^2)$, and
suppose $\cH$ is a $k$-uniform, $n$-vertex, $D$-regular hypergraph, with codegree at most $D^{1 - \delta}$, such that $D \geq \log^{\eps^{-2}}n$. 
By Observation~\ref{obs:incidence-hypergraph}\ref{incidence-hypergraph:uniformity}--\ref{incidence-hypergraph:degree}, $\cH^* \coloneqq \hgInc_D(\cH)$ is a $(k + 1)$-uniform, $D$-regular hypergraph with codegree at most $D^{1-\delta}$, and $e(\cH^*) = D\cdot e(\cH) = D^2 n / k \leq \exp(D^{\eps^2})$.  
Let $\cF \coloneqq \{[D]\times\{v\} : v \in V(\cH)\}$, and note that $\cF$ is a collection of subsets of $V(\cH^*)$ such that $|\cF| \leq n \leq \exp(D^{\eps^2})$ and every $S \in \cF$ satisfies $\sum_{v\in S}d_{\cH^*}(v) = D^2$. 
Thus, if $D$ is sufficiently large, then by Theorem~\ref{thm:egj}, there exists a matching $M$ in $\cH^*$ such that $|S \cap V(M)| \geq (1 - D^{-\eps})|S|$ for every $S \in \cF$.  By Observation~\ref{obs:incidence-hypergraph}\ref{incidence-hypergraph:colouring}, $M_1, \dots, M_D$, where $M_i \coloneqq \{e \in \cH : \exists f \in M,~f\supseteq \{i\}\times e\}$, are pairwise edge-disjoint matchings, and moreover, by the construction of $\cF$, every $v\in V(\cH)$ is covered by all but $D^{1 - \eps}$ of these matchings.  In particular, $\chi'(\cH') \leq D$ and $\Delta(\cH\setminus \cH') \leq D^{1 - \eps}$ where $\cH' \coloneqq \bigcup_{i=1}^D M_i$.  Hence, $\chi'(\cH\setminus \cH') \leq k(\Delta(\cH\setminus \cH')- 1)+ 1 \leq kD^{1 - \eps}$, so $\chi'(\cH) \leq \chi'(\cH') + \chi'(\cH\setminus \cH') \leq D + kD^{1 - \eps} = D + o(D)$, as desired. 

Theorem~\ref{thm:egj} is actually proved via a generalization of Theorem~\ref{thm:molloy-reed} (which implies Theorem~\ref{thm:pippenger-spencer}).  Thus, the above argument is based on `circular logic', but it demonstrates that in the setting of Theorem~\ref{thm:pippenger}, the existence of nearly perfect pseudorandom hypergraph matchings is in some sense equivalent to the existence of a nearly optimal proper edge-colouring (the above comments about the proof of Theorem~\ref{thm:egj} also apply to the result of Alon and Yuster~\cite{AY2005} on pseudorandom matchings, which is proved via Theorem~\ref{thm:pippenger-spencer}).  Moreover,  Kahn's~\cite{kahn1996asymptotically} proof of Theorem~\ref{thm:asym-list-edge-col} (in the case when all lists are the same) more closely resembles the approach described here, wherein a nibble process is used to construct a matching in the incidence hypergraph, than it does Pippenger and Spencer's~\cite{PippengerSpencer} proof of Theorem~\ref{thm:pippenger-spencer}. 

Note that one could also prove Theorem~\ref{thm:egj} 
`more directly' by a more careful analysis of the proof of Theorem~\ref{thm:pippenger} -- the reason being essentially that the matchings chosen in each step of the nibble intersect the sets in $\cF$ as one would expect a random set would.  This intuition is made rigorous in~\cite{KKMO2020}, where Theorem~\ref{thm:kkmo-colouring} is derived from~\cite[Theorem~7.1]{KKMO2020}, a pseudorandom version of Theorem~\ref{thm:kkmo}.   
Moreover, the approach of finding an edge-colouring via Theorem~\ref{thm:egj} and Observation~\ref{obs:incidence-hypergraph} is very versatile and was used e.g.~in~\cite{KKMO2020, KKKMO2021} (see Section~\ref{efl:small}).

\section{Independent sets and vertex-colouring}\label{section:vertex-colouring}

\subsection{Independence number}\label{subsection:independence-number}

Prior to R\"odl's~\cite{rodl1985} proof of the Erd\H os-Hanani conjecture~\cite{ErdosHanani}, in 1981, Ajtai, Koml\'os, and Szemer\'edi~\cite{ajtai1981dense} employed a similar semi-random approach to show that every triangle-free graph has a large independent set.

\begin{theorem}[Ajtai, Koml{\'{o}}s, and Szemer\'edi~\cite{AKS1980, ajtai1981dense}]
    \label{thm:tri-free-ind-set-bound}
    There exists an absolute constant $c > 0$ such that the following holds.  If $G$ is an $n$-vertex triangle-free graph of average degree at most $d$, then
  \begin{equation*}
      \alpha(G) \geq c\left(\frac{n}{d}\right)\log d.
  \end{equation*}
\end{theorem}

This result has spawned intensive research over the last four decades.  Theorem~\ref{thm:tri-free-ind-set-bound} and a hypergraph analogue of it due to Koml{\'o}s, Pintz, Spencer, and Szemer{\'e}di~\cite{KPS82}, have surprising applications to number theory and geometry, respectively.  Improving and generalizing Theorem~\ref{thm:tri-free-ind-set-bound} is also a problem of major importance within combinatorics, in part due to connections to Ramsey theory and to the study of random graphs and algorithms.

In~\cite{ajtai1981dense}, Ajtai, Koml{\'{o}}s, and Szemer\'edi used Theorem~\ref{thm:tri-free-ind-set-bound} to construct an infinite \textit{Sidon sequence} (i.e.~a sequence of positive integers in which the pairwise sums are all distinct) with `high density'; in particular, for every $n$, the sequence contains $\Omega((n\log n)^{1/3})$ integers less than $n$.  Erd\H os conjectured that for every $\eps > 0$, there exists an infinite Sidon sequence containing $\Omega(n^{1/2 - \eps})$ integers less than $n$, and this problem is still open.  The best known result is due to Rusza~\cite{Ruzsa98}, who proved the weaker version with $1/2$ replaced with $\sqrt 2 - 1$ in the exponent, and Cilleruelo~\cite{C14} provided an explicit construction of such a sequence.

A new proof of Theorem~\ref{thm:tri-free-ind-set-bound} was given in~\cite{AKS1980} by Ajtai, Koml{\'{o}}s, and Szemer\'edi (written by Spencer), which uses the Cauchy-Schwarz inequality to build the independent set deterministically, rather than with a random nibble process.  Theorem~\ref{thm:tri-free-ind-set-bound} is used in~\cite{AKS1980} to prove the \textit{Ramsey number} bound $R(3, k) = O(k^2 / \log k)$.  (The Ramsey number $R(\ell, k)$ is the smallest $n$ such that every red-blue edge-colouring of the $n$-vertex complete graph contains either a red copy of $K_\ell$ or a blue copy of $K_k$.)  The matching lower bound $R(3,k) = \Omega(k^2 / \log k)$ was later established by Kim~\cite{kim1995_ramsey}, also using a semi-random approach.  Theorem~\ref{thm:tri-free-ind-set-bound} was improved by Shearer~\cite{Sh83, Sh91}, who showed that the constant $c$ can be replaced with $1 - o(1)$ in Theorem~\ref{thm:tri-free-ind-set-bound}, as conjectured by Ajtai, Koml{\'{o}}s, and Szemer\'edi~\cite{ajtai1981dense}.  Although Shearer's proof is more similar to the Cauchy-Schwartz approach of~\cite{AKS1980} than the random nibble approach of~\cite{ajtai1981dense}, his proof implies that the random greedy algorithm produces an independent set with expected size at least $(1 - o(1))(n / d)\log d$ in every $n$-vertex triangle-free graph of average degree $d$.

Improving the value of the leading constant in Theorem~\ref{thm:tri-free-ind-set-bound}, or determining if $1 - o(1)$ is best possible, is an interesting open problem.  Bollob\'{a}s~\cite{B81} proved that there are $n$-vertex $d$-regular triangle-free graphs with $\alpha(G) \leq 2(n/d) \log d$ (by considering random $d$-regular graphs), so Shearer's bound~\cite{Sh83} is within a factor of about at most two of best possible.  Shearer's result also implies that $R(3, k) \leq (1 + o(1))k^2 / \log k$, which is still the best known upper bound, and any further improvement to the value of $c$ in Theorem~\ref{thm:tri-free-ind-set-bound} would improve this bound on the Ramsey number as well and be a major breakthrough.  Fiz Pontiveros, Griffiths, and Morris~\cite{PGM20} and independently Bohman and Keevash~\cite{BK21} showed that $R(3, k) \geq (1/4 - o(1))k^2/\log k$, so Shearer's bound on $R(3, k)$ is also within a factor of about at most four of best possible. 

Theorem~\ref{thm:tri-free-ind-set-bound} holds more generally for $k$-uniform hypergraphs where $k \geq 2$, as follows.  An \textit{$\ell$-cycle} in a $k$-uniform hypergraph is a set of $\ell$ edges spanned by at most $\ell(k - 1)$ vertices, which does not contain an $\ell'$-cycle for $\ell' < \ell$, and the \textit{girth} of a $k$-uniform hypergraph is the length of its shortest cycle (or infinity if there is no cycle).  
In 1982, Koml{\'o}s, Pintz, Spencer, and Szemer{\'e}di~\cite{KPS82} proved an analogue of Theorem~\ref{thm:tri-free-ind-set-bound} for $3$-uniform hypergraphs of girth at least five and used this result to disprove Heilbronn's conjecture on the Heilbronn triangle problem, which asks for the minimum area of a triangle formed by any three points out of a set of $n$ points placed in the unit disk.  Heilbronn conjectured that this area is at most $O(n^{-2})$ for any set of $n$ points, but Koml{\'o}s, Pintz, Spencer, and Szemer{\'e}di~\cite{KPS82} used their hypergraph analogue of Theorem~\ref{thm:tri-free-ind-set-bound} to construct a set of $n$ points in which the minimum area of a triangle with its vertices among those points is at least $\Omega(n^{-2}\log n)$.
Ajtai, Koml{\'o}s, Pintz, Spencer, and Szemer{\'e}di~\cite{AKPSS82} later generalized the result of~\cite{KPS82} by showing that every $k$-uniform hypergraph on $n$ vertices with girth at least five and average degree at most $d$ contains an independent set of size at least $\Omega(n(\log d / d)^{1 / ( k - 1)})$.  
Duke, Lefmann, and R{\"o}dl~\cite{DLR95} strengthened this result by showing that for $k \geq 3$, this bound holds for hypergraphs of girth at least three (that is, for linear hypergraphs), confirming a conjecture of Spencer~\cite{S90} in a strong sense.  
Note that for $k = 2$, this bound matches the one in Theorem~\ref{thm:tri-free-ind-set-bound}.  
Notably, the proofs in~\cite{KPS82} and~\cite{AKPSS82} use a random nibble approach like in~\cite{ajtai1981dense}.  The proof in~\cite{DLR95} proceeds by a reduction to the case of hypergraphs of girth at least 5, whence the result follows from the result in~\cite{AKPSS82}.

Ajtai, Erd{\H{o}}s, Koml{\'o}s, and Szemer{\'e}di~\cite{AEKS81} suggested that Theorem~\ref{thm:tri-free-ind-set-bound} may still hold for $K_r$-free graphs for any fixed $r$ (and it may even hold more generally for vertex-colouring -- see Conjecture~\ref{conj:fm}), and they proved the weaker result that $K_r$-free graphs on $n$ vertices of average degree at most $d$ have an independent set of size at least $\Omega((n / d)\log\log d)$\COMMENT{The bound $c(n / d)\log((\log d) / r)$ actually holds for any $r$ for some absolute constant $c$.}.  Later, a breakthrough of Shearer~\cite{Sh95} in 1995 improved this bound to $\Omega((n / d)\log d / \log \log d)$, which, up to the leading constant factor, is still the best known.
 Alon~\cite{alon1996independence} proved that Theorem~\ref{thm:tri-free-ind-set-bound} holds more generally for graphs where the neighbourhood of every vertex has bounded chromatic number.  These results of Shearer~\cite{Sh95} and of Alon~\cite{alon1996independence} actually bound the average size of an independent set.  In this vein, Davies, Jenssen, Perkins, and Roberts~\cite{DJPR18} recently proved that the average size of an independent set in a triangle-free graph of maximum degree at most $\Delta$ is at least $(1 - o(1))(n/\Delta)\log\Delta$, which also generalizes Theorem~\ref{thm:tri-free-ind-set-bound} and even matches the earlier bound of Shearer~\cite{Sh83} for the special case of regular graphs\COMMENT{since a graph is regular if and only if its average degree is the same as its maximum degree}.

\subsection{Chromatic number}\label{subsection:chromatic-number}

Nearly all of the results bounding the independence number mentioned in the previous subsection can be generalized to bounds on the chromatic number.  In 1995, Kim~\cite{kim1995} proved that every graph of girth at least five and maximum degree at most $\Delta$ has (list) chromatic number at most $(1 + o(1))\Delta / \log \Delta$.  Independently, Johansson~\cite{johansson1996} proved that every triangle-free graph of maximum degree at most $\Delta$ has chromatic number at most $O(\Delta / \log \Delta)$, which generalizes Theorem~\ref{thm:tri-free-ind-set-bound}.\COMMENT{The result only directly implies $\alpha(G) = \Omega((n / \Delta)\log\Delta)$, but you can always pass to a subgraph with maximum degree at most twice the average degree of the original graph, on at least half the vertices, and you only lose four in the constant factor.} 
In 2019, Molloy~\cite{M17} simultaneously generalized both Kim's~\cite{kim1995} and Johansson's~\cite{johansson1996} result by improving the leading constant in Johansson's result to match that of Kim, as follows.
    
\begin{theorem}[Molloy~\cite{M17}]
\label{thm:tri-free-chi-bound}
    For every $\eps > 0$, there exists $\Delta_0$ such that the following holds for every $\Delta \geq \Delta_0$.
    If $G$ is a triangle-free graph of maximum degree at most $\Delta$, then
    \begin{equation*}
    \chi_\ell(G) \leq (1 + \eps)\frac{\Delta}{\log \Delta}.
    \end{equation*}
\end{theorem}
    
Theorem~\ref{thm:tri-free-chi-bound} also matches Shearer's bound~\cite{Sh83} for regular graphs.  Improving the leading constant in Theorem~\ref{thm:tri-free-chi-bound}, or determining if it is best possible, is another major open problem.  By the same argument as in the previous subsection, the bound in Theorem~\ref{thm:tri-free-chi-bound} is within a factor of at most two of best possible.  In fact, Frieze and {\L}uczak~\cite{FL92} proved that random $\Delta$-regular graphs have chromatic number $(1/2 \pm o(1))\Delta / \log \Delta$ with high probability, and it is an open problem whether there is a polynomial-time algorithm which almost surely finds a proper vertex-colouring of such a graph with at most $(1 - \eps)\Delta / \log\Delta$ colours for some $\eps > 0$ (see~\cite{M17}).  Since random regular graphs of bounded degree have $O(1)$ cycles with high probability, the affirmative would follow if there exists such an algorithm for colouring triangle-free graphs of maximum degree at most $\Delta$ (again, see~\cite{M17}).  A related longstanding open problem of Karp~\cite{K76} is whether there exists a polynomial-time algorithm for finding an independent set of size within a factor two of best possible in a binomial random graph.

The proofs of Kim~\cite{kim1995} and Johansson~\cite{johansson1996} use a nibble approach inspired by Kahn's~\cite{kahn1996asymptotically} proof of Theorem~\ref{thm:asym-list-edge-col}, in which a small random selection of vertices are assigned a colour randomly in each step of the nibble.  Johansson~\cite{johansson1996} never published his proof, but Molloy and Reed~\cite[Chapters~12 and 13]{MR02} provided simpler proofs of the results of both Kim~\cite{kim1995} and Johansson~\cite{johansson1996}\COMMENT{in particular by showing that the nibble process produces a partial colouring which can be completed with a single application of the Lov\'asz Local Lemma by invoking a result of Reed~\cite{R99}, as in~\cite{MR2000}}.  Molloy's~\cite{M17} proof of Theorem~\ref{thm:tri-free-chi-bound}, which uses the `entropy compression' method, is even simpler, and Bernshteyn~\cite{B17} simplified this proof further by showing that the `Lopsided Local Lemma' can be used instead of `entropy compression'.  However, Bernshteyn's proof is non-constructive, and Molloy's `entropy compression' argument provides an efficient randomized algorithm for finding a proper colouring using $(1 + o(1))\Delta / \log \Delta$ colours, matching the `algorithmic barrier' for colouring random graphs described above.  Molloy's proof has inspired further algorithmic results such as in~\cite{AIS19, DKPS20alg}\COMMENT{In particular, the ideas have led to constructive versions of the Local Lemma that apply in settings beyond the `variable model' originally considered by Moser and Tardos~\cite{MT10} such as in applications of the Lopsided Local Lemma.}.

All of these proofs rely on a `coupon collector'-type approach.  Roughly speaking, this means that a useful heuristic is to consider a random colouring, where each vertex $v\in V(G)$ is assigned a colour uniformly at random from a set of colours $C$.  If $G$ is triangle-free, then $G[N(v)]$ is an independent set for every $v\in V(G)$ and is thus properly coloured.  Moreover, the well-known solution to the coupon collector's problem implies that if $d(v) \leq (1 - o(1))|C|\log|C|$, then there is a colour in $C$ not assigned to a neighbour of $v$, which we could potentially use to `recolour' $v$.  In particular, if $G$ has maximum degree $\Delta$ and $|C| \geq (1 + o(1))\Delta / \log \Delta$, then with non-zero probability, for every vertex $v \in V(G)$, less than $|C|$ colours are assigned to a vertex in $N(v)$.  This is of course not sufficient to prove Theorem~\ref{thm:tri-free-chi-bound} but is a useful intuition for the bound.

It is also believed that at the expense of a worse leading constant, Theorem~\ref{thm:tri-free-chi-bound} holds for $K_r$-free graphs for every fixed $r$, as follows.
\begin{conjecture}\label{conj:fm}
    For every $r \in \mathbb N$, there exists a constant $c_r$ such that the following holds.  If $G$ is a $K_r$-free graph with maximum degree at most $\Delta$, then $\chi_\ell(G) \leq c_r \Delta / \log \Delta$.
\end{conjecture}
The resulting bound on the independence number is already a major open problem proposed earlier by Ajtai, Erd{\H{o}}s, Koml{\'o}s, and Szemer{\'e}di~\cite{AEKS81} (as mentioned in Section~\ref{subsection:independence-number}) and is still open even for $r = 4$, and the resulting bound on the chromatic number was conjectured by Alon, Krivelevich, and Sudakov~\cite{AKS99}.
In this direction, Johansson~\cite{J96-Kr} proved that for every fixed $r$, every $K_r$-free graph of maximum degree at most $\Delta$ has list chromatic number $O(\Delta \log \log \Delta / \log \Delta)$, which generalizes the result of Shearer~\cite{Sh95} mentioned at the end of Section~\ref{subsection:independence-number}.  Johansson also proved that for every fixed $r$, if $G$ is a graph of maximum degree at most $\Delta$ that satisfies $\chi(G[N(v)]) \leq r$ for every $v\in V(G)$, then $\chi_\ell(G) = O(\Delta / \log \Delta)$, generalizing the result of Alon~\cite{alon1996independence} mentioned at the end of Section~\ref{subsection:independence-number}.  These results of Johansson were also not published, but Molloy~\cite{M17} gave a new proof of the former, and the latter was proved (using the approach of Bernshteyn~\cite{B17}) by Bonamy, Kelly, Nelson, and Postle~\cite{BKNP18}.
Alon, Krivelevich, and Sudakov~\cite{AKS99} generalized Johansson's result to `locally sparse graphs' by proving the following: if $G$ is a graph of maximum degree at most $\Delta$ such that the neighbourhood of any vertex spans at most $\Delta^2 / f$ edges, then $\chi(G) = O(\Delta / \log \sqrt f)$ for $f \leq \Delta^2 + 1$, and Vu~\cite{Vu02} generalized this result to list colouring.  Davies, Kang, Pirot, and Sereni~\cite{DKPS20} improved this result by showing that it holds with a leading constant of $1 + o(1)$ as $f\rightarrow\infty$, thus generalizing Theorem~\ref{thm:tri-free-chi-bound}\COMMENT{by setting $f = \Delta^2 + 1$}.

\begin{theorem}[Davies, Kang, Pirot, and Sereni~\cite{DKPS20}]\label{thm:DKPS}
   For every $\eps > 0$, there exists $\Delta_0$ such that the following holds for every $\Delta \geq \Delta_0$.  If $G$ is a graph of maximum degree at most $\Delta$ such that the neighbourhood of any vertex spans at most $\Delta^2 / f$ edges for $f \leq \Delta^2 + 1$, then 
   \begin{equation*}
       \chi_\ell(G) \leq (1 + \eps)\frac{\Delta}{\log \sqrt f}.
   \end{equation*}
\end{theorem}

We note that the aforementioned results of Kim~\cite{kim1995}, Johansson~\cite{johansson1996, J96-Kr}, and Vu~\cite{Vu02} all use the nibble method.  Davies, Kang, Pirot, and Sereni~\cite{DKPS20} provided a generalization of all of these results (and also Theorem~\ref{thm:DKPS}) by introducing the `local occupancy method'.  This method reduces these colouring problems to optimization problems involving relevant local properties of the `hard-core model', which is a family of probability distributions over the independent sets of a graph with origins in statistical physics.  Their approach builds on the work of Molloy~\cite{M17} and Bernshteyn~\cite{B17} and subsequent work in~\cite{BKNP18, DdJdVKP20}, and the approach used to prove the results of~\cite{Sh95, alon1996independence, DJPR18} bounding the average size of independent sets mentioned in the previous subsection and also of~\cite{DJPR17} may be viewed as a precursor to these methods.  The main result of Davies, Kang, Pirot, and Sereni~\cite{DKPS20} is proved using the Lopsided Local Lemma as in Bernshteyn's~\cite{B17} proof of Theorem~\ref{thm:tri-free-chi-bound}.  It can also be proved using entropy compression as in the original proof of Theorem~\ref{thm:tri-free-chi-bound} of Molloy~\cite{M17}, and indeed, Davies, Kang, Pirot, and Sereni~\cite{DKPS20alg} used this approach to obtain additional algorithmic coloring results.

All of the results mentioned so far in this subsection provide a bound of $o(\Delta)$ on the chromatic number of graphs of maximum degree $\Delta$ under a `local sparsity' condition.  Trivially, every graph $G$ satisfies $\chi(G) \leq \Delta(G) + 1$, and Brooks~\cite{B41} famously showed that equality holds if and only if $G$ is a complete graph or an odd cycle (when $G$ is connected).  With a considerably relaxed `local sparsity' condition, we can still bound the chromatic number away from $\Delta$, as in the following result.

\begin{theorem}[Molloy and Reed~\cite{MR02}] \label{local-sparsity-lemma}
 For every $\zeta > 0$, there exists $\Delta_0$ such that the following holds for every $\Delta \geq \Delta_0$.  If $G$ is a graph of maximum degree at most $\Delta$ and every $v\in V(G)$ satisfies $|E(G[N(v)])| \leq (1 - \zeta)\binom{\Delta}{2}$, then $\chi(G) \leq (1 - \zeta/e^6)\Delta$.
\end{theorem}

This result was improved by Bruhn and Joos~\cite{BJ15} and by Bonamy, Perrett, and Postle~\cite{BPP18}.  Recently, Hurley, de Joannis de Verclos, and Kang~\cite{HdVK20} improved it further by proving the bound $\chi(G) \leq (1 - \zeta/2 + \zeta^{3/2}/6 + o(1))\Delta$, which gives the correct dependence on $\zeta$ as $\zeta \rightarrow 0$.  Determining the best possible bound in Theorem~\ref{local-sparsity-lemma} for larger $\zeta$ is an interesting problem; any further improvements would also improve the best known bound for Reed's $\omega$, $\Delta$, $\chi$ conjecture~\cite{R98} and for the Erd\H{o}s-Ne{\v{s}}et{\v{r}}il conjecture~\cite{EN85}.
We also use Theorem~\ref{local-sparsity-lemma} in our proof of the Erd\H os-Faber-Lov\' asz conjecture (see Section~\ref{efl:large}), but we do not need the improvements of~\cite{BJ15, BPP18, HdVK20}.
The following related problem was posed by Vu~\cite{Vu02} in 2002.

\begin{conjecture}[Vu~\cite{Vu02}]\label{conj:vu}
For every $\zeta, \eps > 0$, there exists $\Delta_0$ such that the following holds for every $\Delta \geq \Delta_0$. If $G$ is a graph of maximum degree at most $\Delta$ and every two distinct vertices have at most $\zeta \Delta$ common neighbours in $G$, then $\chi_\ell (G) \leq (\zeta + \eps) \Delta$.
\end{conjecture}
This conjecture is still open if we replace $\chi_\ell(G)$ by $\chi(G)$, and even the much weaker conjecture, that $G$ satisfies $\alpha(G) \geq (1/\zeta - \eps)(n / \Delta)$ is still open\COMMENT{it is even not known if $\alpha(G) \geq K^{-1}(1/\zeta - o(1))|V(G)| / D$ for some absolute constant $K$}.  The results of~\cite{HdVK20} give nontrivial bounds when $\zeta$ is close to one.
If true, Conjecture~\ref{conj:vu} with $\zeta = 1/k$ implies Theorem~\ref{thm:asym-list-edge-col} for linear hypergraphs, as follows.  Let $\cH$ be a $k$-bounded linear hypergraph with maximum degree at most $D$. It is clear that $\Delta(L(\cH)) \leq kD$, and every two distinct vertices in $L(\cH)$ have at most $\max\{k^2 , (D-2) + (k-1)^2\} \leq \zeta k (D + k^2)$ common neighbours.  Letting $\Delta \coloneqq k(D + k^2)$, Conjecture~\ref{conj:vu} would imply $\chi'(\cH) = \chi(L(\cH)) \leq (\zeta + \eps)\Delta = D + o(D)$ when $k$ is fixed as $D\rightarrow \infty$.  Recently, Kelly, K\"{u}hn, and Osthus~\cite{KKO21} confirmed a special case of Conjecture~\ref{conj:vu} that also recovers this application to Theorem~\ref{thm:asym-list-edge-col}.

Our final problem on vertex-colouring graphs is the following conjecture of Alon and Krivelevich~\cite{AK98} from 1998 on the list chromatic number of bipartite graphs.

\begin{conjecture}[Alon and Krivelevich~\cite{AK98}]
    There exists $K$ such that the following holds.  If $G$ is a bipartite graph of maximum degree at most $\Delta$, then $\chi_\ell(G) \leq K\log \Delta$. 
\end{conjecture}

The best known bound for this conjecture is provided by Theorem~\ref{thm:tri-free-chi-bound}; however, this bound can also be proved more directly with the `coupon collector' argument described earlier.%
\COMMENT{Let $(A, B)$ be the bipartition of $G$, and let $C$ be a list assignment for $G$ in which each vertex has $(1 + \eps)\Delta / \log \Delta$ available colours.  For each vertex $a \in A$, assign $a$ a colour uniformly at random from its list.  For each $b \in B$, the probability that a colour $c \in C(b)$ is not assigned to any neighbour of $b$ is at least $(1 - \log \Delta / ((1 + \eps)\Delta))^{d(b)} \geq \Delta^{-1/(1 + \eps)}$, and thus, there are at least $\Delta^{\eps / 2}$ such colours in expectation.  Standard arguments involving concentration inequalities and the Local Lemma imply that there is an outcome in which for every vertex $b \in B$, there is a colour $c \in C(b)$ not assigned to a neighbour of $b$.  In such an outcome, the colouring can immediately be extended to $B$.}
Alon, Cambie, and Kang~\cite{ACK21} used this argument to prove a stronger result for list colouring bipartite graphs when each vertex in one of the parts has a list of available colours of the conjectured size.
Alon and Krivelevich~\cite{AK98} also suggested that the stronger bound $\chi_\ell(G) \leq (1 + o(1))\log_2 \Delta$ may also hold, which would be best possible for complete bipartite graphs.  In fact, Saxton and Thomason~\cite{ST15} proved that every graph of minimum degree at least $d$ has list chromatic number at least $(1 - o(1))\log_2 d$, improving an earlier result of Alon~\cite{A00}.  


\subsection{Hypergraph colourings}\label{subsection:vtx-col-hyper}

Theorem~\ref{thm:tri-free-ind-set-bound} can not only be generalized to vertex-colouring in the graphic setting but also for hypergraphs.  In 2013, Frieze and Mubayi~\cite{FM13} proved the following result, which generalizes both Johansson's theorem~\cite{johansson1996} and Theorem~\ref{thm:tri-free-ind-set-bound}. 

\begin{theorem}[Frieze and Mubayi~\cite{FM13}]\label{thm:fm-chi-bound}
  For every $k \geq 2$ there exists $c, \Delta_0 > 0$ such that the following holds for every $\Delta \geq \Delta_0$.  If $\cH$ is a $k$-uniform hypergraph with maximum degree at most $\Delta$ and girth at least four, then 
  \begin{equation*}
      \chi(\cH) \leq c\left(\frac{\Delta}{\log\Delta}\right)^{\frac{1}{k - 1}}.
  \end{equation*}
\end{theorem}

To prove this result, Frieze and Mubayi~\cite{FM13} analyzed a nibble procedure inspired by the proof of Johansson~\cite{johansson1996}.
Molloy~\cite{MolloyCanaDAM} conjectured that for $k = 3$ the result holds for $c = \sqrt 2 + o(1)$, as this value is suggested by the `coupon collector' heuristic described in Section~\ref{subsection:chromatic-number}\COMMENT{If each vertex of $\cH$ is assigned a colour from a set $C$ independently and uniformly at random, then for every $v \in V(\cH)$ and $c \in C$, the probability that there is no edge $e \ni v$ in $\cH$ such that both vertices of $e \setminus \{v\}$ are assigned $c$ is at least $(1 - 1 / |C|^2)^{d_\cH(v)} \geq e^{-\Delta / |C|^2}$, and thus, if $|C| =\sqrt{(2 + \eps)\Delta / \log \Delta}$, then this probability is at least $\Delta^{-1/(2 + \eps)}$.  In particular, there are at least $|C|\Delta^{-1/(2 + \eps)} = \Omega(\Delta^{\eps} / \sqrt{\log \Delta})$ such colours in expectation.}, and he asked more broadly if it can also be proved with either the `entropy compression' or the `local occupancy' approach.  Iliopoulos~\cite{I19} showed that the bound $\chi_\ell(\cH) \leq (1 + o(1))(k - 1)(\Delta / \log \Delta)^{1 / (k - 1)}$ holds in Theorem~\ref{thm:fm-chi-bound} if $\cH$ has girth at least five.

For $k \geq 3$, Frieze and Mubayi~\cite{FM13} `bootstrapped' Theorem~\ref{thm:fm-chi-bound} to show that it actually holds for linear hypergraphs (that is, hypergraphs of girth at least three), by applying it with a distinct set of colours to vertex-disjoint induced subgraphs of girth at least four whose vertices partition $V(\cH)$.  This latter result generalizes the bound of Duke, Lefmann, and R{\"o}dl~\cite{DLR95} on the independence number mentioned in Section~\ref{subsection:independence-number} to vertex-colouring.
Cooper and Mubayi~\cite{CM15} also generalized Theorem~\ref{thm:fm-chi-bound} for $k = 3$ by showing that the girth hypothesis can be replaced with the condition that $\cH$ has no \textit{triangle}, where a triangle is a set of three edges $e$, $f$, and $g$ such that there exist vertices $u$, $v$, and $w$ satisfying $\{u, v\}\subseteq e$, $\{v, w\}\subseteq f$, $\{u, w\} \subseteq g$, and $\{u, v, w\}\cap e \cap f \cap g = \varnothing$.\COMMENT{They actually showed that $\chi(\cH) = O(\max\{\sqrt{\Delta_3 / \log \Delta_3}, \Delta_2 / \log \Delta_2\})$ for any triangle-free $3$-bounded hypergraph, where $\Delta_i$ is the maximum number of size-$i$ edges containing any fixed vertex for $i\in\{2,3\}$.}  
Cooper and Mubayi~\cite{CM16} later showed that both of these results hold under more general `local sparsity' conditions similar to that of Theorem~\ref{thm:DKPS} for graphs.  Frieze and Mubayi~\cite{FM08, FM13} conjectured a generalization of Conjecture~\ref{conj:fm} for $k$-uniform hypergraphs; however, Cooper and Mubayi~\cite{CM2017} disproved this conjecture for all $k \geq 3$.

\section{The Erd\H{o}s-Faber-Lov\'asz conjecture}\label{section:EFL}

In this section we introduce and provide background for the Erd\H os-Faber-Lov\' asz conjecture, which we abbreviate to the EFL conjecture.  
Earlier developments related to the EFL conjecture are also detailed in the surveys of Kahn~\cite{kahn1995asymptotics, kahn1997} and of Kayll~\cite{kayll2016}.
The EFL conjecture states the following (recall that a hypergraph is linear if it has codegree one):\COMMENT{As we discuss later, there are several equivalent ways to state this conjecture, but we refer to this statement as `the EFL conjecture'.}

\begin{enumerate}[label={(EFL\arabic*)}, topsep = 6pt]
\item\label{EFL}Every $n$-vertex linear hypergraph has chromatic index at most $n$.
\end{enumerate}

Erd\H{o}s often wrote that this was one of his `three favourite combinatorial problems' (see e.g.,~\cite{kahn1997}). 
Erd\H{o}s, Faber, and Lov\'asz famously formulated this conjecture at a tea party in 1972.
 The simplicity and elegance of the EFL conjecture initially led them to believe it would be easily solved (see e.g., the discussion in~\cite{cg1998} and~\cite{erdos1981}). 
However, as the difficulty became apparent Erd\H{o}s offered successively increasing rewards for a proof of the conjecture, which eventually reached \$500. 

The following three infinite families of hypergraphs are extremal for this conjecture (see Figure~\ref{fig:extremal}):
\begin{figure}
  \begin{minipage}[b]{.33\linewidth}
    \centering
    \includegraphics{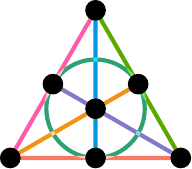}\\
    Projective plane%
  \end{minipage}%
  \begin{minipage}[b]{.34\linewidth}
    \centering
    \includegraphics{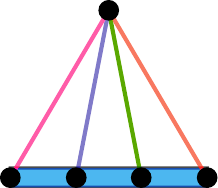}\\
    Degenerate plane%
  \end{minipage}%
  \begin{minipage}[b]{.33\linewidth}
    \centering
    \includegraphics{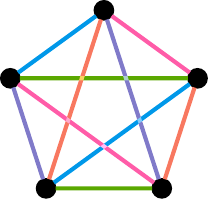}\\
    $K_{2t + 1}$%
  \end{minipage}%
  \caption{Extremal examples for the Erd\H os-Faber-Lov\' asz conjecture}
  \label{fig:extremal}
\end{figure}
\begin{itemize}
    \item finite projective planes of order $k$ (known to exist when $k$ is a prime power), which are $(k + 1)$-uniform, linear, intersecting hypergraphs on $n$ vertices with $n$ edges where $n \coloneqq k^2 + k + 1$;
    \item degenerate planes, also called `Near Pencils', which are linear, intersecting hypergraphs on $n$ vertices with $n$ edges for any $n \in \mathbb N$ consisting of one edge of size $n - 1$ and $n - 1$ edges of size two; and
    \item complete graphs on $n$ vertices where $n\in\mathbb N$ is odd (as well as some `local' modifications of these).
\end{itemize}
The vastly different structure of these extremal examples contributes to the difficulty of the EFL conjecture.  Note in particular that the first two examples have edges of unbounded size as $n \rightarrow \infty$, whereas complete graphs are $2$-uniform.\COMMENT{The first two examples have chromatic index $n$ (matching the upper bound of the EFL conjecture), because any intersecting hypergraph $\cH$ satisfies $\chi'(\cH) = |\cH|$.  To see that the third example has chromatic index $n$, first note that if $G \cong K_n$, then $\chi'(G) \leq n$ by Vizing's theorem~\cite{vizing1965}.  Moreover, $\chi'(G) > n - 1$ if $n$ is odd since $G$ does not contain a perfect matching and in any proper edge-colouring of an $(n - 1)$-regular graph with $(n - 1)$ colours, each colour class is a perfect matching.}
Let us note that we can (and will) assume without loss of generality that hypergraphs have no edges of size one in the EFL conjecture, since any proper edge-colouring of the edges of  size at least two in an $n$-vertex linear hypergraph $\cH$ with at most $n$ colours can be extended to the remaining size-one edges of $\cH$ (again with at most $n$ colours).  Without this assumption, the hypergraph obtained from an $n$-vertex star by adding the edge of size one containing the center vertex is also an extremal example.

\subsection{Equivalent formulations}

Part of the beauty of this conjecture lies in the fact that it can be equivalently stated in several simple, yet seemingly unconnected, ways.  The following are all in fact equivalent to the EFL conjecture:  
\begin{enumerate}[label={(EFL\arabic*)},  topsep = 6pt]\stepcounter{enumi}
\item\label{EFL-hg-dual} If $\cH$ is a linear hypergraph with $n$ edges, each of size at most $n$, then the vertices of $\cH$ can be coloured with at most $n$ colours such that no edge contains two vertices of the same colour. 

\item\label{EFL-set-theoretic} If $A_1 , \dots , A_n$ are sets of size $n$ such that every pair of them shares at most one element, then the elements of $\bigcup_{i=1}^{n} A_i$ can be coloured with $n$ colours so that all colours appear in each $A_i$.

\item\label{EFL-line-graph} If $G_1 , \dots , G_n$ are complete graphs, each on at most $n$ vertices, such that $|V(G_i) \cap V(G_j)| \leq 1$ for every $1 \leq i \ne j \leq n$, then the chromatic number of $\bigcup_{i=1}^{n}G_i$ is at most $n$.
\end{enumerate}
\COMMENT{Let us note that every $n$-vertex linear hypergraph has maximum degree at most $n$ (the linearity condition, combined with the fact that every vertex of $\cH$ is contained in at most one edge of size one, implies that each vertex has degree at most $(n - 1) + 1 = n$).}  
We show that the EFL conjecture is equivalent to~\ref{EFL-hg-dual}--\ref{EFL-line-graph} by showing the following implications: \ref{EFL}~$\Rightarrow$~\ref{EFL-hg-dual}~$\Rightarrow$~\ref{EFL-set-theoretic}~$\Rightarrow$~\ref{EFL-line-graph}~$\Rightarrow$~\ref{EFL}.

For the first implication, we need to introduce the notion of \textit{hypergraph duality}.  The \textit{dual} of a hypergraph $\cH$ is the hypergraph $\cH^*$ with vertex set $\cH$ and edge set $\{ \{e \ni v: e \in \cH\}: v \in V(\cH)\}$  (see Figure~\ref{fig:dual} for an example).  Clearly, the dual of $\cH^*$ is isomorphic to $\cH$ itself\COMMENT{since the map $\{e \ni v: e \in \cH\} \mapsto v$ is a bijection from $V((\cH^*)^*)$ to $V(\cH)$ that preserves the edges, where $(\cH^*)^*$ is the dual of $\cH^*$.}.  Note that $\cH$ is linear if and only $\cH^*$ is linear.  Now suppose $\cH$ is a linear hypergraph with $n$ edges, each of size at most $n$.  We may assume without loss of generality that every vertex of $\cH$ has degree at least two\COMMENT{since we can remove any vertex of degree at most one, deleting it from any edge containing it, and extend any $n$-colouring of the result graph to one of $\cH$}.  Since $\cH$ has $n$ edges and is linear, $\cH^*$ has $n$ vertices and is also linear, so~\ref{EFL} implies that there is proper edge-colouring of $\cH^*$ using at most $n$ colours.  By assigning each vertex of $\cH$ the colour of the corresponding edge of $\cH^*$, we obtain the desired colouring, proving~\ref{EFL-hg-dual}.

To show that \ref{EFL-hg-dual} $\Rightarrow$ \ref{EFL-set-theoretic}, let $\cH$ be the hypergraph with vertex set $\bigcup_{i=1}^n A_i$ and edge set $\{A_i : i \in [n]\}$.  Since $A_1, \dots, A_n$ have size $n$ and every pair of them shares at most one element, $\cH$ is linear with $n$ edges, each of size $n$.  By~\ref{EFL-hg-dual}, the vertices of $\cH$ can be coloured with at most $n$ colours such that no edge contains two vertices of the same colour.  Since every edge has size $n$, every edge contains a vertex of every colour, so this colouring satisfies~\ref{EFL-set-theoretic}.
To prove \ref{EFL-set-theoretic} $\Rightarrow$ \ref{EFL-line-graph}, first note that by possibly adding new vertices to each $G_i$, we may assume without loss of generality that $|V(G_i)| = n$ for each $i \in [n]$.  Letting $A_i = V(G_i)$ for each $i \in [n]$,~\ref{EFL-set-theoretic} implies there is a colouring of $\bigcup_{i=1}^n A_i$ with $n$ colours so that all colours appear in each $A_i$.  In particular, if $u, v\in A_i$, then $u$ and $v$ are assigned different colours, so this colouring is also a proper vertex-colouring of $\bigcup_{i=1}^n G_i$, proving~\ref{EFL-line-graph}.

Finally, to prove that the EFL conjecture follows from~\ref{EFL-line-graph}, let $\cH$ be a linear hypergraph on $n$ vertices, and for each $v \in V(\cH)$, let $G_v$ be the complete graph with vertex set $\{e \ni v : e \in \cH\}$.  Since $\cH$ is linear, each $G_v$ for $v\in V(\cH)$ has at most $n$ vertices, and $|V(G_u)\cap V(G_v)| \leq 1$ for distinct $u, v\in V(\cH)$.  Since $\bigcup_{v \in V(\cH)} G_v$ is in fact the line graph of $\cH$ (see Figure~\ref{fig:dual}), we have $\chi'(\cH) = \chi(\bigcup_{v \in V(\cH)} G_v)$, so \ref{EFL-line-graph} $\Rightarrow$ \ref{EFL}, as desired.

An interpretation of~\ref{EFL-line-graph} in terms of a scheduling problem was given by Haddad and Tardif~\cite{HT04}.

\begin{figure}
  \begin{minipage}[b]{.33\linewidth}
    \centering
    \includegraphics{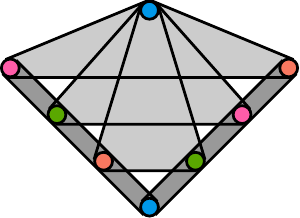}%
  \end{minipage}%
  \begin{minipage}[b]{.34\linewidth}
    \centering
    \includegraphics{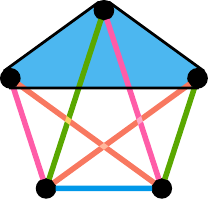}%
  \end{minipage}%
  \begin{minipage}[b]{.33\linewidth}
    \centering
    \includegraphics{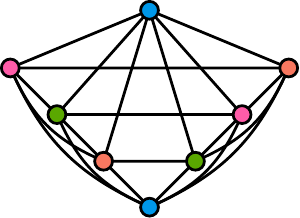}%
  \end{minipage}%
  \caption{The hypergraph dual (left) of the $5$-vertex hypergraph in the center, and the line graph (right).}
  \label{fig:dual}
\end{figure}

\subsection{Results}\label{efl:results}

Recently, we confirmed the EFL conjecture for all but finitely many hypergraphs.

\begin{theorem}[Kang, Kelly, K\"uhn, Methuku, and Osthus~\cite{KKKMO2021}]\label{thm:efl}
    For every sufficiently large $n$, every $n$-vertex linear hypergraph has chromatic index at most $n$.
\end{theorem}

The proof of Theorem~\ref{thm:efl} can be turned into a randomized polynomial-time algorithm.
The necessary modifications are discussed in detail in~\cite{KKKMO2021_focs}.  
We also proved the following `stability result', predicted by Kahn~\cite{kahn1995asymptotics}.
\begin{theorem}[Kang, Kelly, K\"uhn, Methuku, and Osthus~\cite{KKKMO2021}]\label{main-thm-2}
    For every $\delta > 0$, there exist $n_0, \sigma > 0$ such that the following holds for every $n \geq n_0$. If $\cH$ is an $n$-vertex linear hypergraph such that 
    \begin{enumerate}[(i)]
        \item\label{stability:degree} $\cH$ has maximum degree at most $(1-\delta)n$ and
        \item\label{stability:pp} the number of edges of size $(1 \pm \delta)\sqrt{n}$ in $\cH$ is at most $(1-\delta)n$,
    \end{enumerate} 
    then the chromatic index of $\cH$ is at most $(1-\sigma)n$.
\end{theorem}

The hypothesis~\ref{stability:degree} in Theorem~\ref{main-thm-2} ensures that $\cH$ does not too closely resemble the degenerate plane or the complete graph, while~\ref{stability:pp} ensures that $\cH$ does not too closely resemble a projective plane, since projective planes on $n$ vertices have $n$ edges of size roughly $\sqrt n$.

Let us overview previous progress leading up to these results.  
Predating the EFL conjecture, in 1948 de Bruijn and Erd\H{o}s~\cite{bruijn_erdos1948} showed that every intersecting $n$-vertex linear hypergraph has at most $n$ edges\COMMENT{They also proved that every intersecting $n$-vertex hypergraph with $n$ edges is either a degenerate plane or a projective plane.}.  Equivalently, the line graph of an $n$-vertex linear hypergraph contains no clique of size greater than $n$.  Seymour~\cite{seymour1982} proved that every $n$-vertex linear hypergraph $\cH$ contains a matching of size at least $|\cH| / n$, which implies the de Bruijn-Erd\H{o}s theorem, as an intersecting hypergraph has matching number one.  Kahn and Seymour~\cite{KS1992} strengthened this result by proving that every $n$-vertex linear hypergraph has fractional chromatic index at most $n$\COMMENT{since every hypergraph $\cH$ satisfies $|\cH| / \nu(\cH) \leq \chi'_f(\cH)$}.  (Recall that every hypergraph $\cH$ satisfies $\chi'_f(\cH) \leq \chi'(\cH)$, so all of these results are relaxations of the EFL conjecture).   Chang and Lawler~\cite{chang1988} proved that every $n$-vertex linear hypergraph has chromatic index at most $\lceil3n/2 - 2\rceil$.  

Interestingly, results from both Section~\ref{section:edge-colouring} and Section~\ref{section:vertex-colouring} have the following immediate applications to the EFL conjecture, which are illustrative to note.
\begin{enumerate}[label=(\theequation)]
    \stepcounter{equation}
    \item\label{thm:PS-EFL} For every $\eps, k > 0$, there exists $n_0$ such that the following holds for every $n \geq n_0$.  If $\cH$ is a $k$-bounded, $n$-vertex, linear hypergraph, then $\chi'(\cH) \leq n + \eps n$ and moreover, if every $e \in \cH$ satisfies $|e| \geq 3$, then $\chi'(\cH) \leq n$.
    \stepcounter{equation}
    \item\label{thm:AKS-EFL} For every $\eps > 0$, there exist $\delta, n_0 > 0$ such that the following holds for every $n \geq n_0$ and $k \coloneqq \delta \sqrt n$.  If $\cH$ is a $k$-uniform, $n$-vertex, linear hypergraph, then $\chi'(\cH) \leq \eps n$.
    \stepcounter{equation}
    \item\label{thm:MR-EFL} For every $\eps > 0$, there exist $\delta, n_0 > 0$ such that the following holds for every $n \geq n_0$ and $k \coloneqq (1 - \delta)\sqrt n$.  If $\cH$ is a $k$-uniform, $n$-vertex, linear hypergraph, then $\chi'(\cH) \leq (1 - \eps)n$.
\end{enumerate}
To prove~\ref{thm:PS-EFL}, it suffices to note that since $\cH$ is linear, it has maximum degree at most $n / (\min_{e\in\cH}|e| - 1)$) (assuming $\cH$ has no size-one edges), whence~\ref{thm:PS-EFL} follows immediately from Theorem~\ref{thm:asym-list-edge-col}.  To prove~\ref{thm:AKS-EFL} and~\ref{thm:MR-EFL}, it suffices to note that since $\cH$ is linear, the line graph $L \coloneqq L(\cH)$ has maximum degree at most $k(n - k)/(k - 1) \leq (1 + 2/k)n$ and every pair of adjacent vertices in $L$ have at most $(k - 1)^2 + n / k$ common neighbours.  Hence, \ref{thm:AKS-EFL} follows immediately from Theorem~\ref{thm:DKPS},\COMMENT{since $L$ has maximum degree at most $n + 2\delta^{-1}\sqrt n$ and every neighbourhood of $L$ spans at most $n(\delta^2n + \delta^{-1}\sqrt n)$ edges} and~\ref{thm:MR-EFL} follows immediately from Theorem~\ref{local-sparsity-lemma}.\COMMENT{since $L$ has maximum degree at most $n + 3\sqrt n$ and every neighbourhood of $L$ spans at most $\binom{n + 3\sqrt n}{2}((1 - \delta)^2 + 3 / \sqrt n)$ edges}  

In 1992, a breakthrough by Kahn~\cite{kahn1992coloring} confirmed the EFL conjecture asymptotically, by showing that every $n$-vertex linear hypergraph has chromatic index at most $n + o(n)$.  Note that this result strengthens the first part of~\ref{thm:PS-EFL} by showing the $k$-boundedness assumption is not necessary. 
Kahn's argument in~\cite{kahn1992coloring} relies on a `restricted' list colouring result which strengthens the Pippenger-Spencer theorem (Theorem~\ref{thm:pippenger-spencer}) but is still weaker than Theorem~\ref{thm:asym-list-edge-col}, and thus can be viewed as a `stepping stone' towards Theorem~\ref{thm:asym-list-edge-col}.  Moreover,
Kahn's argument from~\cite{kahn1992coloring}, combined with Theorem~\ref{thm:asym-list-edge-col} that he proved later in~\cite{kahn1996asymptotically}, can be adapted to prove that every $n$-vertex linear hypergraph has \textit{list} chromatic index at most $n + o(n)$, which we explain further in Section~\ref{section:asym-list-EFL}.
The second part of~\ref{thm:PS-EFL} was strengthened in 2019 by Faber and Harris~\cite{FH19}, who proved that for some absolute constant $c$, the EFL conjecture holds if every edge has size at least three and at most $c\sqrt n$.  In fact, the main result of~\cite{FH19} also implies~\ref{thm:AKS-EFL}.  Their argument relies on Theorem~\ref{thm:asym-list-edge-col} and the result of Vu~\cite{Vu02} mentioned before Theorem~\ref{thm:DKPS}.  That~\cite{AKS99, Vu02} have applications to the EFL conjecture was first observed by Faber~\cite{faber2010}, namely to prove a result similar to~\ref{thm:AKS-EFL}.


Nevertheless, none of the results prior to Theorem~\ref{thm:efl} confirmed the conjecture for any nontrivial class of hypergraphs containing one of the extremal families.  In particular, the case of $k$-bounded (or even $3$-bounded) hypergraphs was still open (and was highlighted as a challenging problem by Kahn~\cite{kahn1995asymptotics}).  Similarly, the case of hypergraphs in which all edges have size $\omega(1)$ was also still open. 
Both of these cases turned out to be significant stepping stones towards the proof of Theorem~\ref{thm:efl}, and their proofs contain several of the main ideas.  To highlight these ideas, we provide a detailed sketch of the following two results in Section~\ref{section:EFL-proof}.

\begin{restatable}{theorem}{smallEFL}
\label{thm:small-efl}
    There exists $n_0 > 0$ such that the following holds.  If $\cH$ is an $n$-vertex linear hypergraph such that every $e \in \cH$ satisfies $|e| \in \{2, 3\}$ and $n > n_0$, then $\chi'(\cH) \leq n + 1$.
\end{restatable}

\begin{restatable}{theorem}{largeEFL}
\label{large-edge-thm}
    For every $\delta > 0$, there exist $n_0, r, \sigma > 0$ such that the following holds.
    If $\mathcal H$ is an $n$-vertex linear hypergraph where $n > n_0$ and every $e\in \cH$ satisfies $|e| > r$, then $\chi'(\cH) \leq n$.  Moreover, if $\chi'(\cH) > (1 - \sigma)n$, then $|\{e \in \cH : |e| = (1 \pm \delta)\sqrt n\}| \geq (1 - \delta)n$.
\end{restatable}

Notice that in Theorem~\ref{thm:small-efl}, the bound on the chromatic index is one larger than in the EFL conjecture.
In Section~\ref{efl:small}, we prove Theorem~\ref{thm:small-efl} and briefly explain the additional ideas required to prove the stronger bound of the EFL conjecture in this case.  The proof of Theorem~\ref{thm:small-efl} can also be adapted with little additional effort to prove the same result for $k$-bounded hypergraphs, for any fixed $k$.  We focus on the case of $3$-bounded hypergraphs as it is slightly cleaner yet complex enough to capture many of the important ideas.  

Note also that Theorem~\ref{large-edge-thm} implies Theorem~\ref{main-thm-2} in the case when all edges of $\cH$ are sufficiently large.  This `stability result' is needed to combine the arguments of Theorem~\ref{thm:small-efl} and~\ref{large-edge-thm} to obtain Theorem~\ref{thm:efl}.  Roughly, we apply (a stronger version of) Theorem~\ref{large-edge-thm} first to the `large' edges of $\cH$, and then we apply the arguments of Theorem~\ref{thm:small-efl} to find a proper edge-colouring of the `small' edges of $\cH$ that is compatible with the colouring of the `large' edges.  If Theorem~\ref{large-edge-thm} only requires $(1 - \sigma) n$ colours, then only minor adaptations to the arguments of Theorem~\ref{thm:small-efl} are required, which we briefly describe in Section~\ref{efl:large} after proving Theorem~\ref{large-edge-thm}.  If Theorem~\ref{large-edge-thm} only guarantees a proper edge-colouring of the large edges of $\cH$ with $n$ colours, then additional ideas are required, for which we refer the interested reader to~\cite{KKKMO2021} (which in particular contains a sketch of the overall argument).

\subsection{Open problems}

 We now discuss some open problems related to the EFL conjecture.  
 First, it would be interesting to characterize when equality holds in Theorem~\ref{thm:efl}.  As mentioned, finite projective planes, degenerate planes, and complete graphs on an odd number of vertices are extremal examples.  In fact, any $n$-vertex hypergraph with more than $(n - 1)^2 / 2$ size-two edges has chromatic index at least $n$ when $n$ is odd\COMMENT{since $\chi'(G) \geq \Gamma(G) \geq 2|E(G)| / (|V(G)| - 1| > n - 1$}, and any hypergraph $\cH$ obtained from $K_n$ by replacing a complete subgraph with a single edge $e$ is linear and has chromatic index $n$ if $|V(\cH)\setminus e|$ is odd\COMMENT{Let $X \coloneqq V(\cH)\setminus e$.  Suppose for a contradiction that there is a proper edge-colouring using $n - 1$ colours, and let $c$ denote the colour assigned to $e$.  Since every vertex in $X$ has degree $n - 1$, every vertex in $X$ is contained in an edge assigned $c$, and moreover, the other end of this edge is in $X$.  Thus, the edges assigned $c$ form a perfect matching in $X$, contradicting that $|X|$ is odd.}  (note that the degenerate plane is obtained in this way).  These may include all of the extremal examples.

Berge~\cite{berge1989} and F{\"{u}}redi~\cite{furedi1986} independently posed the following beautiful conjecture.
 \begin{conjecture}[Berge~\cite{berge1989}, F{\"{u}}redi~\cite{furedi1986}]\label{conj:bf}
    If $\cH$ is a linear hypergraph with vertex set $V$, then $\chi'(\cH) \leq \max_{v\in V} |\bigcup_{e\ni v}e|$.
 \end{conjecture}
 If true, Conjecture~\ref{conj:bf} implies the EFL conjecture, since every linear hypergraph $\cH$ satisfies $\max_{v\in V(\cH)} |\bigcup_{e\ni v}e| \leq n$.  Note also that if $G_\cH$ is the graph obtained from $\cH$ by replacing each edge $e\in\cH$ with a complete graph on the vertices of $e$ (sometimes called the \textit{shadow} of $\cH$), then $|\bigcup_{e\ni v}e| = \Delta(G_\cH) + 1$.  In particular, if $\cH$ is $2$-uniform, then $G_\cH = \cH$, so Conjecture~\ref{conj:bf} if true, also implies Vizing's theorem.  The fractional relaxation of Conjecture~\ref{conj:bf} is still open.  The following related conjecture was posed by F{\"u}redi, Kahn, and Seymour~\cite{FKS93}: If $\cH$ is a multi-hypergraph with vertex set $V$, then $\chi'_f(\cH) \leq \max_{v \in V}\sum_{e\ni v}(|e| - 1 + 1/|e|)$.  This conjecture may even hold for the chromatic index, which if true, would generalize Shannon's theorem~\cite{S49} and imply Conjecture~\ref{conj:alon-kim} in the case $k = t$.

 It is also natural to ask whether the EFL conjecture holds more generally for list colouring.  Faber~\cite{F17} conjectured that it does, as follows.

\begin{conjecture}[The `List' EFL conjecture~\cite{F17}]\label{conj:list-efl}
    Every $n$-vertex linear hypergraph has list chromatic index at most $n$.
\end{conjecture}

Conjecture~\ref{conj:list-efl} was recently confirmed by the authors~\cite{KKKMO2021b} for the special case of hypergraphs of maximum degree at most $n - o(n)$, and their result also implies that in this case projective planes are the only extremal examples.  The main result in~\cite{KKKMO2021b} also solves a conjecture of Erd\H{o}s on the chromatic index of hypergraphs of small codegree.

A related problem to Conjecture~\ref{conj:list-efl} is an algebraic strengthening of the EFL conjecture involving the Combinatorial Nullstellensatz, posed by Janzer and Nagy~\cite{JN20}.

As mentioned, the arguments of Kahn~\cite{kahn1992coloring, kahn1996asymptotically} can be adapted to prove that the List EFL conjecture holds asymptotically.  Kahn's proof in~\cite{kahn1992coloring} also implies that Conjecture~\ref{conj:bf} holds asymptotically.  In fact, assuming the sizes of the lists are polylogarithmic in $n$, it is easy to show that the argument can be adapted to work for list colouring, as follows.

\begin{theorem}\label{thm:asym-list-bf}
  For every $\eps > 0$, there exists $n_0$ such that the following holds for every  $n, D \geq n_0$.  If $\cH$ is an $n$-vertex linear hypergraph such that $|\bigcup_{e\ni v}e| \leq D$ for every $v \in V(\cH)$, then $\chi'(\cH) \leq (1 + \eps)D$.  Moreover, if $D \geq \log^2 n$, then $\chi'_\ell(\cH) \leq (1 + \eps)D$.
\end{theorem}

For completeness, we prove Theorem~\ref{thm:asym-list-bf} in Section~\ref{section:asym-list-EFL}.  It would be interesting to prove that the bound on the list chromatic index in Theorem~\ref{thm:asym-list-bf} holds without the assumption $D \geq \log^2 n$.\COMMENT{This would follow if one could find the set $R\subseteq \bigcup_{e\in\cH}C(e)$ satisfying~\eqref{list-reserved-colours}.  A natural approach would be to apply the Lov\'asz Local Lemma instead of the Union Bound.  However, if $e, f \in \cH$ satisfy $C(e) \cap C(f) \neq \varnothing$, then the trials determining if $c \in R$ for each $c \in C(e)\cap C(f)$ affects both $|R\cap C(e)|$ and $|R\cap C(f)|$.  Thus, the `dependencies' cannot be bounded by a function of $D$.}

The next open problem is the following special case of a conjecture of Larman.

\begin{conjecture}[`Restricted' Larman's conjecture]\label{conj:larman}
    If $\cH$ is an $n$-vertex intersecting hypergraph, then there exists a decomposition of $\cH$ into $\cF_1, \dots, \cF_n \subseteq \cH$ such that $|F \cap F'| \geq 2$ for every $F, F' \in \cF_i$ and $i \in [n]$.
\end{conjecture}

The full version of Larman's conjecture was a combinatorial relaxation of Borsuk's conjecture from 1933, which states that every set of diameter at most one in $\mathbb R^d$ can be partitioned into at most $d + 1$ sets of diameter strictly less than one.  However, in 1993, Kahn and Kalai~\cite{KK93} disproved Larman's conjecture (and thus in turn Borsuk's conjecture).  Nevertheless, they asked if the special case of Larman's conjecture presented in Conjecture~\ref{conj:larman} still holds (see also~\cite{kalai_2015}), in part because of its resemblance to the EFL conjecture.   

Finally, we note that Alon, Saks, and Seymour (see Kahn~\cite{kahn1991}) conjectured the following `bipartite version' of~\ref{EFL-line-graph}: If a graph $G$ can be decomposed into $k$ edge-disjoint bipartite graphs, then the chromatic number of $G$ is at most $k+1$. This conjecture was a generalization of the Graham--Pollak theorem~\cite{gp1972} on edge decompositions of complete graphs into bipartite graphs, which has applications to communication complexity. However, it was disproved by Huang and Sudakov~\cite{HS2012} in a strong form, i.e., the conjectured bound on the chromatic number is far from being true.
\subsection{Asymptotic list colouring version of the Berge-F\"uredi conjecture}\label{section:asym-list-EFL}

In this subsection we prove Theorem~\ref{thm:asym-list-bf}.  We only prove the bound on the list chromatic index when $D \geq \log^2 n$, as the proof of the general bound on the chromatic index will be evident from the argument we provide here.\COMMENT{To prove the general bound on the chromatic index, we apply the same argument, assuming $C(e) = [\lceil(1 + \eps)D\rceil]$ for every $e \in \cH$.  Letting $R \coloneqq [\lceil \gamma(1 + \eps)D\rceil]$,~\eqref{list-reserved-colours} holds and the rest of the proof remains the same.}  Our proof closely follows the approach of~\cite{kahn1992coloring}, but with a simple additional trick, and using Theorem~\ref{thm:asym-list-edge-col} instead of~\cite[Theorem 1.3]{kahn1992coloring}.

\begin{proof}[Proof of Theorem~\ref{thm:asym-list-bf}]
  Let
  \begin{equation*}
    1 / n_0 \ll 1 / r_0 \ll 1 / r_1 \ll \gamma \ll \eps \ll 1,
  \end{equation*}
  let $n \geq n_0$, let $D \geq \log^2 n$, and let $\cH$ be an $n$-vertex linear hypergraph such that $|\bigcup_{e\ni v}e| \leq D$ for every $v \in V(\cH)$.  It suffices to show that if $C$ is an assignment of lists $C(e)$ to every $e \in \cH$, such that every $e \in \cH$ satisfies $|C(e)| \geq (1 + \eps)D$, then $\cH$ has a proper edge-colouring $\phi$ such that $\phi(e) \in C(e)$ for every $e \in \cH$.  We assume without loss of generality that $|C(e)| = (1 + \eps)D \pm 1$.

  Let $\ordering$ be a linear ordering of the edges of $\cH$ satisfying $e \ordering f$ if $|e| > |f|$, and decompose $\cH$ into the following spanning subhypergraphs:
  \begin{itemize}
  \item $\cH_{\mathrm{sml}} \coloneqq \{ e \in \cH \: : \: |e| \leq r_1 \}$,
  \item $\cH_{\mathrm{med}} \coloneqq \{ e \in \cH \: : \: r_1 < |e| \leq r_0 \}$, and 
  \item $\cH_{\mathrm{lrg}} \coloneqq \{ e \in \cH \: : \: |e| > r_0 \}$.
  \end{itemize}
  Since $\cH$ is linear, 
  \begin{enumerate}[label=(\theequation)]
    \stepcounter{equation}
  \item\label{eqn:fwd-deg-bound} every $e \in \cH_{\mathrm{lrg}}$ satisfies $|\{f \in N_\cH(e) : f \ordering e\}| \leq |e| (D - |e|) / (|e| - 1) \leq (1 + \eps / 3) D$,
    \stepcounter{equation}
  \item\label{eqn:med-max-deg} $\Delta(\cH_{\mathrm{med}}) \leq D / r_1$ (and thus, by Theorem~\ref{thm:asym-list-edge-col}, $\chi'_\ell(\cH_{\mathrm{med}}) \leq 2 D / r_1 \leq \gamma D / 2$), and
    \stepcounter{equation}
  \item\label{eqn:sml-lrg-intersect} every $e \in \cH_{\mathrm{sml}}$ satisfies $|N_\cH(e) \cap \cH_{\mathrm{lrg}}| \leq r_1 D / r_0 \leq \eps D / 4$.
  \end{enumerate}

  Now we show there exists a set $R \subseteq \bigcup_{e\in\cH}C(e)$ such that
  \begin{equation}\label{list-reserved-colours}
    \text{every $e \in \cH$ satisfies $|R \cap C(e)| = (1 \pm 1/2)\gamma |C(e)|.$}
  \end{equation}
  Include every colour in $R$ randomly and independently with probability $\gamma$.  By a standard application of the Chernoff bound, every $e \in \cH$ satisfies $|R \cap C(e)| = (1\pm 1/2)\gamma |C(e)|$ with probability at least $1 - 2\exp(-\gamma |C(e)| / 12) \geq 1 - 2\exp(- \gamma \log^2 n / 12)$, so by the Union Bound,
  ~\eqref{list-reserved-colours} holds with high probability, so there indeed exists such a set $R$.

  Fix $R$ satisfying~\eqref{list-reserved-colours}, and for every $e \in \cH$, let $C'(e) \coloneqq C(e) \setminus R$ and $R(e) \coloneqq C(e) \cap R$. 
  By~\eqref{list-reserved-colours},
    \begin{enumerate}[label=(\theequation)]
    \stepcounter{equation}
  \item\label{eqn:unreserved-col-lb} every $e \in \cH_{\mathrm{lrg}}\cup \cH_{\mathrm{sml}}$ satisfies $|C'(e)| \geq (1 + \eps / 2)D$, and
    \stepcounter{equation}
  \item\label{eqn:reserved-col-lb} every $e \in \cH_{\mathrm{med}}$ satisfies $|R(e)| \geq \gamma D / 2$.
    \end{enumerate}
    Therefore, by~\ref{eqn:fwd-deg-bound} and~\ref{eqn:unreserved-col-lb}, there exists a proper edge-colouring $\phi_{\mathrm{lrg}}$ of $\cH_{\mathrm{lrg}}$ such that $\phi_{\mathrm{lrg}}(e) \in C'(e)$ for every $e \in \cH_{\mathrm{lrg}}$, and by~\ref{eqn:med-max-deg} and~\ref{eqn:reserved-col-lb}, there exists a proper edge-colouring $\phi_{\mathrm{med}}$ of $\cH_{\mathrm{med}}$ such that $\phi_{\mathrm{med}}(e) \in R(e)$ for every $e\in \cH_{\mathrm{med}}$.  Now for each $e \in \cH_{\mathrm{sml}}$, let $C''(e) \coloneqq C'(e)\setminus \{\phi_{\mathrm{lrg}}(f) : f \in N_\cH(e) \cap \cH_{\mathrm{lrg}}\}$.  By~\ref{eqn:sml-lrg-intersect} and~\ref{eqn:unreserved-col-lb}, $|C''(e)| \geq (1 + \eps / 4)D$ for every $e \in \cH_{\mathrm{sml}}$.  Therefore, by Theorem~\ref{thm:asym-list-edge-col}, there exists a proper edge-colouring $\phi_{\mathrm{sml}}$ of $\cH_{\mathrm{sml}}$ such that $\phi_{\mathrm{sml}}(e) \in C''(e)$ for every $e \in \cH_{\mathrm{sml}}$.  By combining $\phi_{\mathrm{lrg}}$, $\phi_{\mathrm{med}}$, and $\phi_{\mathrm{sml}}$, we obtain the desired colouring.  
\end{proof}

\section{Proving the Erd\H os-Faber-Lov\' asz conjecture}\label{section:EFL-proof}

In this section we give detailed sketches of the proofs of Theorems~\ref{thm:small-efl} and~\ref{large-edge-thm}, the special cases of the proof of the EFL conjecture in~\cite{KKKMO2021} discussed in Section~\ref{efl:results}.  

\subsection{Using $n + 1$ colours when edge-sizes are bounded}\label{efl:small}

We begin with Theorem~\ref{thm:small-efl}, which we restate for the readers' convenience.

\smallEFL*

In this subsection, we fix constants satisfying the hierarchy
\begin{equation}\label{eqn:small-efl-hierarchy}
    0 < 1/n_0  \ll \xi \ll \kappa \ll \gamma \ll \eps \ll 1,
\end{equation}
we let $n \geq n_0$, and we let $\cH$ be an $n$-vertex linear hypergraph such that every $e\in \cH$ satisfies $|e| \in \{2, 3\}$.  We assume without loss of generality that every pair of vertices of $\cH$ is contained in an edge, since otherwise we can add a size-two edge to $\cH$ to obtain an $n$-vertex linear hypergraph with chromatic index greater than or equal to $\chi'(\cH)$.
Let $G$ be the graph with $V(G) \coloneqq V(\cH)$ and $E(G) \coloneqq \{e \in \cH : |e| = 2\}$, and let  $U\coloneqq\{u\in V(\cH) : d_G(u) \geq (1 - \eps)n\}$.  
Since every pair of vertices is contained in precisely one edge, we have
\begin{equation}\label{eq:degree}
    \text{$n - 1 = 2(d_\cH(v) - d_G(v)) + d_G(v) = 2 d_\cH(v) - d_G(v)$ for every vertex $v \in V(\cH)$.}
\end{equation}

\begin{figure}
\begin{minipage}[b]{.5\linewidth}
    \centering
    \includegraphics[scale=1.5]{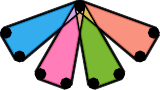}\\\textbf{Low degree:} more flexibility
\end{minipage}%
\begin{minipage}[b]{.5\linewidth}
    \centering
    \includegraphics[scale=1.5]{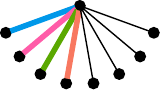}\\\textbf{High degree:} more graph-like
    \end{minipage}  
    \caption{Two partial edge-colourings using $4$ colours (when $n$ = 9, say).  The uncoloured edges form a graph of maximum degree at most $4$ and can be coloured with at most $5$ colours by Vizing's theorem.}
    \label{fig:degrees}
\end{figure}

Our strategy to prove Theorem~\ref{thm:small-efl} is to reduce it to Vizing's theorem.  In order to do that, it suffices to partially colour $\cH$ with $k < n$ colours (for some suitable $k \sim n / 2$) such that every edge of $\cH\setminus E(G)$ is coloured and the remaining uncoloured edges of $G$ form a graph of maximum degree at most $n - k$ (see Figure~\ref{fig:degrees} and Lemma~\ref{lem:graph-reduction}).
Roughly speaking, each colour class of this partial colouring will be obtained by first constructing a large matching via the R\"odl nibble and then extending it to cover (essentially) all of $U$.  The latter step is of course necessary in order to obtain an (uncoloured) leftover graph of small maximum degree. (It is also sufficient since $U$ consists of precisely those $v\in V(\cH)$ with $d_\cH(v)\sim n$).  On the other hand, while this is the reason we need to pay special attention to the vertices in $U$, the definition of $U$ also means that we have many (graph) edges at our disposal, which allow us to carry out the extension step mentioned above.
To make this precise, we introduce the following important definition.

\begin{definition}[Perfect and nearly perfect coverage]
  Let $\cM$ be a set of edge-disjoint matchings in $\cH$, and let $S \subseteq U$.
    
    \begin{itemize}
        \item We say $\cM$ has \emph{perfect coverage} of $U$ if each $M \in \cM$ covers $U$.
        
        \item We say $\cM$ has \emph{nearly perfect coverage of $U$ with defects in $S$} if
        \begin{enumerate}[(i)]
            \item each $u \in U$ is covered by at least $|\cM|-1$ matchings in $\cM$ and
    
            \item each $M \in \cM$ covers all but at most one vertex in $U$, and $U \setminus V(M) \subseteq S$.
        \end{enumerate}
    \end{itemize}
\end{definition}
More precisely, using $k \coloneqq \lceil n/2\rceil + \lceil \gamma^{1/3}n\rceil$ colours, we will partially colour $\cH$ such that
\begin{itemize}
    \item every edge of $\cH\setminus E(G)$ is coloured,
    \item at least $d_G(v) / 2 - 2\xi n$ edges of $G$ containing $v$ are coloured for every $v \in V(G)$, and
    \item the colour classes have nearly perfect coverage of $U$ (with defects in $U$).
\end{itemize}
As we will show, these conditions ensure that the partial colouring can be extended via Vizing's theorem to all of $\cH$ using at most $n + 1$ total colours.

The first step of the proof is to randomly construct a `reservoir' consisting of edges of $G$ (which will be used for the extension step), as in the following lemma.
\begin{lemma}[Reservoir lemma]\label{lemma:reservoir}
  There exists $R\subseteq E(G)$ satisfying the following:
  \begin{enumerate}[label=\textbf{R\arabic*}]
  \item\label{reservoir-typicality} (Typicality) every $v \in V(\cH)$ satisfies
    \begin{equation*}
      \text{$|N_R(v) \cap U| = |N_G(v) \cap U|/2 \pm \xi n$ and $|N_R(v) \setminus U| = |N_G(v) \setminus U|/2 \pm \xi n$;}
    \end{equation*}
  \item\label{reservoir-upper-regularity} (Upper regularity) for every pair of disjoint sets $S, T\subseteq V(\cH)$ with $|S|, |T| \geq \xi n$, we have
    \begin{equation*}
      \pushQED{\qed} 
      |E_G(S, T)\cap R| \leq (1/2 + \xi)|S||T|.
      \qedhere
      \popQED
    \end{equation*}
  \end{enumerate}
\end{lemma}
This lemma can be proved with a straightforward application of the Chernoff Bound and the Union Bound, by considering the set $R$ to be chosen randomly, where each edge of $G$ is included independently and with probability $1/2$, so we omit the details.  For the remainder of the subsection, we fix $R$ satisfying Lemma~\ref{lemma:reservoir}.

By~\ref{reservoir-typicality}, every vertex $v \in V(\cH)$ satisfies $d_{\cH\setminus R}(v) = d_\cH(v) - d_G(v)/2 \pm 2\xi n$.  Hence, by~\eqref{eq:degree},
\begin{equation}\label{non-reserved-is-regular}
    \text{every vertex $v \in V(\cH)$ satisfies } d_{\cH\setminus R}(v) = \frac{n - 1}{2} \pm 2\xi n.
\end{equation}
Note that by~\eqref{non-reserved-is-regular}, the Pippenger-Spencer theorem already implies $\chi'(\cH\setminus R) \leq (1/2 + \gamma^{1/3})n$, but we need to prove the stronger result that there is a set of pairwise edge-disjoint matchings $\cM = \{M_1, \dots, M_k\}$ such that $M_1\cup\cdots\cup M_k \supseteq \cH\setminus R$ and $\cM$ has nearly perfect coverage of $U$ (with defects in $U$).  

\subsubsection{Absorption}

To obtain these matchings with nearly perfect coverage of $U$, we combine the nibble method with an absorption strategy.  We first find matchings in $\cH\setminus R$ covering almost all of $U$ using Theorem~\ref{thm:egj}, and then for each such matching, we find a vertex-disjoint matching in $R$ covering (all but at most one of) the remaining vertices of $U$.  We extend the first matching by adding the matching in $R$, and in this way we `absorb' the uncovered vertices of $U$.  It will be convenient to work with the following definitions, which will apply to the matchings produced by Theorem~\ref{thm:egj}.

\begin{definition}[Pseudorandom matchings]\label{pseudorandom-matching-def}
  For a family $\cF$ of subsets of $V(\cH)$, a matching $M$ in $\cH$ is \textit{$(\gamma, \kappa)$-pseudorandom} with respect to $\cF$ if every $S\in\cF$ satisfies $|S \setminus V(M)| = \gamma|S| \pm \kappa n$.
\end{definition}

\begin{definition}[Absorbable matchings]
  \label{def:absorbable}
  Let $R'\subseteq R$, let $S\subseteq U$, and let $M$ be a matching in $\cH\setminus R$.   We say $(M , R' , S)$ is \textit{absorbable} if
  \begin{enumerate}[label=\textbf{AB\arabic*}]
  \item\label{AB:S-bound} $|S| \geq \min\{|U|, \gamma n\}$,
  \item\label{AB:reservoir-preserved} $\Delta(R\setminus R') \leq \gamma n$, and
  \item\label{AB-pseudo} either $|V(M)| \leq \sqrt\gamma n$, or $M$ is  $(\gamma,\kappa)$-pseudorandom with respect to $\cF(R') \cup \{U,S\}$, where
    \begin{equation*}
      \cF(R') \coloneqq \{N_{R'}(u) \cap U \: : \: u \in U \} \cup \{N_{R'}(u) \setminus U \: : \: u \in U \}.
    \end{equation*}
  \end{enumerate}
  If the former holds in~\ref{AB-pseudo}, we say $(M, R', S)$ is absorbable by the \textit{smallness} of $M$, and if the latter holds, we say $(M, R', S)$ is absorbable by the \textit{pseudorandomness} of $M$.
\end{definition}

In the proof of Theorem~\ref{thm:small-efl}, we apply our absorption argument successively to each matching constructed by the nibble.  Hence, in each step we will consider absorbable tuples $(M, R', S)$ where $M$ was obtained via a nibble process, $R'$ consists of reservoir edges not used in previous absorption steps, and $S$ consists of vertices of $U$ that are not the `defect' from any of the previous absorption steps.
Now we can state our main absorption lemma, but first we note the following proposition, which is used in its proof.  

\begin{proposition}\label{internal-matching-prop}
  Let $0 < 1/m_0 \ll \alpha \ll 1$, and let $m \geq m_0$ be even.
  If $H$ is an $m$-vertex graph such that 
  \begin{enumerate}[(i)]
      \item\label{internal-matching-prop:degree} every $v\in V(H)$ satisfies $d_H(v) \geq 3 m / 8$ and
      \item\label{internal-matching-prop:upper-regularity} every pair of disjoint sets $S, T \subseteq V(H)$ with $|S|, |T| \geq \alpha m$ satisfies $e_H(S, T) \leq (1/2 + \alpha)|S||T|$,
  \end{enumerate}
  then $H$ has a perfect matching.
\end{proposition}
To prove Proposition~\ref{internal-matching-prop}, one can consider a random equitable partition of $V(H)$ and apply Hall's theorem.
\COMMENT{
\begin{proof}
  By a standard probabilistic argument, $H$ has a spanning bipartite subgraph $H'$ with bipartition $(A, B)$ such that $|A| = |B| = m / 2$ and every $v\in V(H)$ satisfies $d_{H'}(v) \geq m / 6$.
  Since every $v\in A$ satisfies $d_{H'}(v) \geq m / 6$, every $S\subseteq A$ satisfies
  \begin{equation*}
    |E_{H'}(S, N(S))| \geq |S| m / 6.
  \end{equation*}
   However, by~\ref{internal-matching-prop:upper-regularity}, we also have
  \begin{equation*}
    |E_{H'}(S, N(S))| \leq (1/2 + \alpha)|S||N(S)|.
  \end{equation*}
  Combining these two inequalities, we have $|N(S)| \geq 3 m / 10$.  Let $T\coloneqq B\setminus N(S)$.  If $T = \varnothing$, then $S$ satisfies the Hall condition, as desired, so we assume $T \neq \varnothing$.  By the same argument applied to $T$, we have $|N(T)| \geq 3m / 10$, so $|S| \leq |A| - |N(T)| \leq m / 5$.  Now $|S| \leq m / 5 \leq 3m / 10 \leq |N(S)|$, as desired.
\end{proof}}

\begin{lemma}[Absorption lemma]\label{absorption-lemma}
  Let $R' \subseteq R$, let $S\subseteq U$, and let $\cN := \{N_1,\dots, N_{k}\}$ be a set of pairwise edge-disjoint matchings in $\cH\setminus R$.
  If either
  \begin{enumerate}[(i)]
  \item\label{absorption-lemma:pseudorandom} $k \leq \kappa n$ and for every $i \in [k]$, $(N_i , R' , S)$ is absorbable by the pseudorandomness of $N_i$, or 
  \item $k \leq \gamma^{1/3} n$ and for every $i\in[k]$, $(N_i , R' , S)$ is absorbable by the smallness of $N_i$,
  \end{enumerate}
    then there is a set of pairwise edge-disjoint matchings $\cM \coloneqq \{M_1, \dots, M_k\}$ in $\cH$ such that
    \begin{itemize}
        \item $M_i \supseteq N_i$ and $M_i \setminus N_i \subseteq R'$ for all $i \in [k]$, and 
        \item $\cM$ has nearly perfect coverage of $U$ with defects in $S$, and moreover, if $|U| < 3n / 4$, then $\cM$ has perfect coverage of $U$.
    \end{itemize}
\end{lemma}
\begin{proof}
    Let $\cF := \cF(R') \cup \{U, S\}$, and for each $i\in [k]$, let $U_i \coloneqq U\setminus V(N_i)$.
    In both cases, the proof proceeds roughly as follows.  If $|U| < n / 100$, then one-by-one for each $i \in [k]$, we can greedily find a matching $N_i^{\mathrm{abs}}$ of edges in $R'$, with precisely one end in $U_i$ and the other end not in $V(N_i)$, edge-disjoint from those previously chosen, that covers $U_i$.\COMMENT{We can do greedy instead of Hall's because our absorber has density 1/2 instead of $\rho$}
    Letting $M_i \coloneqq N_i \cup N_i^{\mathrm{abs}}$ for each $i \in [k]$, $\{M_1, \dots, M_k\}$ has perfect coverage of $U$, as desired.  If $|U| \geq n / 100$, then one-by-one for each $i \in [k]$, using Proposition~\ref{internal-matching-prop}, we can find a matching $N_i^{\mathrm{abs}}$ of edges in $R'$, with both ends in $U_i$, edge-disjoint from those previously chosen, that contain all but at most one vertex of $U_i$.  Moreover, we ensure that the vertices in each $U_i$ not covered by $N_i^{\mathrm{abs}}$ are distinct, and if $|U| < 3n/4$, we can also augment each $N_i^{\mathrm{abs}}$ with an edge of $R'$ that has an end in $V(\cH)\setminus (U \cup V(N_i) \cup V(N_i^{\mathrm{abs}}))$ to cover $U$.  Hence, $\{M_1, \dots, M_k\}$ has nearly perfect coverage of $U$ with defects in $S$ and perfect coverage of $U$ if $|U| < 3n/4$, where $M_i \coloneqq N_i \cup N_i^{\mathrm{abs}}$ for each $i \in [k]$, as desired.  We only provide a formal proof of the case when~\ref{absorption-lemma:pseudorandom} holds and $|U| \geq n / 100$, as this case is the most challenging. 
    
    For each $i \in [k]$, let $G_i$ be the graph with $V(G_i) := U_i$ and $E(G_i) \coloneqq \{ e \in R' \: : \: e \subseteq U_i\}$.  Since $N_i$ is $(\gamma,\kappa)$-pseudorandom with respect to $\cF \ni U$, we have
  \begin{equation}\label{eqn:size_ui}
      \textrm{$|U_i| = \gamma |U| \pm \kappa n$, and in particular, $|U_i| \geq \gamma n / 200$}.
  \end{equation}

    We claim that for each $i\in[k]$ there exists $u_i \in U_i$ and a matching $N^{\mathrm{abs}}_i$ in $G_i$ such that the following holds.  The vertices $u_1, \dots, u_k$ are distinct, the matchings $N^{\mathrm{abs}}_1, \dots, N^{\mathrm{abs}}_k$ are pairwise edge-disjoint, and $N^{\mathrm{abs}}_i$ covers every vertex of $U_i\setminus\{u_i\}$ for each $i\in[k]$.  Moreover, if $|U| < 3n / 4$, then $N^{\mathrm{abs}}_i$ covers every vertex of $U_i$ for each $i\in [k]$, and otherwise $u_i \in S$.
  
  To that end, we choose distinct $u_i \in U_i$ for each  $i\in[k]$, as follows.

  \begin{itemize}
    \item If $|U| \leq 3n / 4$, then by~\ref{reservoir-typicality} and~\ref{AB:reservoir-preserved}, every $u\in U_i$ satisfies $|N_{R'}(u)\setminus U| \geq (1/4 - \eps)n / 2 - \xi n - \gamma n \geq n / 10$.  By~\ref{AB-pseudo}, since $N_i$ is $(\gamma, \kappa)$-pseudorandom with respect to $\cF \supseteq \cF(R')$ for each $i\in[k]$, this inequality implies that every $u \in U_i$ satisfies $|N_{G_i}(u)\setminus U| \geq \gamma n / 20$.  Since $k \le \kappa n$ and $\kappa \ll \gamma$, by \eqref{eqn:size_ui}, we can choose $u_i \in U_i$ one-by-one such that there is a matching $\{u_iv_i : i\in[k]\}$ where $v_i \in N_{G_i}(u_i)\setminus U$ for each $i\in[k]$.  
  
    \item Otherwise,~\ref{AB:S-bound} implies $|S| \geq \gamma n$, and since $N_i$ is $(\gamma, \kappa)$-pseudorandom with respect to $\mathcal F \ni S$, by~\ref{AB-pseudo}, we have $|S \setminus V(N_i)| \geq \gamma |S| - \kappa n \ge \gamma^2 n /2 > \kappa n$ for each $i \in [k]$, so we can choose $u_i \in U_i \cap S = S \setminus V(N_i)$ one-by-one such that they are distinct, as required.
  \end{itemize}
  
  
  Now let $U'_i \coloneqq U_i\setminus\{u_i\}$ if $|U_i|$ is odd. Otherwise, let $U'_i\coloneqq U_i$.
  By the choice of the vertices $u_1, \dots, u_k$, it suffices to find pairwise edge-disjoint perfect matchings ${N'}_i^{\mathrm{abs}}$ in $G_i[U'_i]$ for each $i\in [k]$.  Indeed if $|U| \leq 3n / 4$ and $|U_i|$ is odd, then $N_i^{\rm abs} \coloneqq {N'}_i^{\mathrm{abs}} \cup \{u_i v_i \}$ satisfies the claim, and otherwise $N_i^{\rm abs} \coloneqq {N'}_i^{\mathrm{abs}}$ satisfies the claim.
  
  We find these matchings one-by-one using Proposition~\ref{internal-matching-prop}.  To this end, we assume that for some $\ell \leq k$, we have found such matchings ${N'}^{\mathrm{abs}}_i$ for $i\in[\ell - 1]$, and we show that there exists such a matching ${N'}^{\mathrm{abs}}_{\ell}$, which proves the claim.  Let $G'_\ell \coloneqq G_\ell[U'_\ell] \setminus \bigcup_{i\in[\ell-1]}{N'}^{\mathrm{abs}}_i$.
  Since $|U| \geq n / 100$, by~\ref{reservoir-typicality} and~\ref{AB:reservoir-preserved}, every $u\in U$ satisfies
  \begin{equation}\label{internal-absorption-degree-outside-tmp}
    |N_{R'}(u) \cap U| \geq |N_R(u) \cap U| - \gamma n \geq (|U| - \eps n)/2 - 2\gamma n \geq 49|U| / 100.
  \end{equation}
  
  Note that $N_{\ell}$ is $(\gamma , \kappa)$-pseudorandom with respect to $\cF \supseteq \cF(R') \cup \{U\}$ by~\ref{AB-pseudo}. Together with ~\eqref{internal-absorption-degree-outside-tmp}, this implies that every $u\in U'_{\ell}$ satisfies $d_{G_{\ell}[U'_\ell]}(u) \geq \gamma |N_{R'}(u)\cap U| - \kappa n - 1 \geq  48\gamma |U| / 100$.  Since $\ell \leq k \leq \kappa n$, we have
  \begin{equation}\label{internal-absorption-degree-outside2-tmp}
    d_{G'_{\ell}}(u) \geq d_{G_\ell[U_i']}(u) - \kappa n \geq 47\gamma  |U| / 100.
  \end{equation}
  
  By~\eqref{eqn:size_ui}, we also have
  \begin{equation}\label{internal-absorption-lemma-part-size-bound-tmp}
    |U'_\ell| \pm 1 = |U_\ell| = \gamma |U| \pm \kappa n, \text{ and in particular } |U'_\ell| \leq 5\gamma |U| / 4.
  \end{equation}
  Combining~\eqref{internal-absorption-degree-outside2-tmp} and~\eqref{internal-absorption-lemma-part-size-bound-tmp}, we have $d_{G'_\ell}(u) \geq 3|U'_\ell| / 8$ for every $u \in U'_\ell$. So by~\ref{reservoir-upper-regularity} and~\eqref{internal-absorption-lemma-part-size-bound-tmp}, we can apply Proposition~\ref{internal-matching-prop} to $G'_\ell$ with $200\xi/\gamma$ as $\alpha$\COMMENT{$|U'_\ell| \overset{\eqref{internal-absorption-lemma-part-size-bound-tmp}}{\ge} \gamma n/100 - \kappa n -1 \ge \gamma n/200$, so by \ref{reservoir-upper-regularity}, $G'_\ell$ satisfies Proposition~\ref{internal-matching-prop}\ref{internal-matching-prop:upper-regularity}, and by~\eqref{internal-absorption-degree-outside2-tmp}, $G'_\ell$ satisfies Proposition~\ref{internal-matching-prop}\ref{internal-matching-prop:degree}.}, to obtain a perfect matching ${N'}_\ell^{\mathrm{abs}}$ in $G'_\ell$, as desired.  
  
  
  Therefore we have pairwise edge-disjoint matchings $N_i^{\rm abs}$ in $G_i$, as claimed, which by construction are edge-disjoint from $N_1 , \dots , N_k$. For each $i \in [k]$, let $M_i \coloneqq N_i\cup N_i^{\mathrm{abs}}$, and let $\cM = \{M_1, \dots, M_k\}$. Now $M_i \supseteq N_i$ and $M_i\setminus N_i\subseteq R$ for each $i\in[k]$, and $\cM$ has nearly-perfect coverage of $U$ with defects in $S$, as desired.  Moreover, if $|U| < 3n/4$, then $\cM$ has perfect coverage of $U$, as desired.
\end{proof}

\subsubsection{Finding absorbable matchings}

Lemma~\ref{absorption-lemma} allows us to apply absorption for up to $\kappa n$ `pseudorandom' matchings at a time. We construct these collections of matchings in the following lemma using Theorem~\ref{thm:egj} and the strategy described in Section~\ref{edge-colouring:pseudo}.

\begin{lemma}[Nibble Lemma]\label{lem:pseudorandom_matching}
  Let $D \in [n^{1/2} , n]$, and let $\cH' \subseteq \cH$ be a spanning subhypergraph such that for every $w \in V(\cH')$ we have $d_{\cH'}(w) = (1 \pm \sqrt\xi)D$.  If $\cF_V$ and $\cF_E$ are families of subsets in $V(\cH')$ and $E(\cH')$, respectively, such that $|\cF_V|,|\cF_E|  \leq n^{\log n}$,
then there exist pairwise edge-disjoint matchings $N_1 , \dots , N_D$ in $\cH'$ such that 
\begin{enumerate}[label=\textbf{N\arabic*}]
    \item\label{pseudorandom_matching:vtx} $N_i$ is $(\gamma, \kappa)$-pseudorandom with respect to $\cF_V$ for every $i \in [D]$, and
    
    \item\label{pseudorandom_matching:edge} $|F \setminus \bigcup_{i=1}^{D} N_i| \leq \gamma |F| + \kappa \max(|F|,D)$ for each $F \in \cF_E$.
\end{enumerate}
\end{lemma}
\begin{proof}[Proof (sketch)]
     First we embed $\cH'$ into a 3-uniform linear hypergraph $\cH''$ with $O(n^4)$ vertices, in which every vertex has degree $(1 \pm \sqrt \xi)D$.  We then let $\cH^* \coloneqq \hgInc_D(\cH'')$ be the $D$-wise incidence hypergraph of $\cH''$ (recall Definition~\ref{def:incidence-hypergraph}).  By Observation~\ref{obs:incidence-hypergraph}\ref{incidence-hypergraph:uniformity}--\ref{incidence-hypergraph:degree}, $\cH^*$ is 4-uniform, linear, and every vertex of $\cH^*$ has degree $(1 \pm \sqrt \xi)D$.  Thus, we can apply Theorem~\ref{thm:egj} to $\cH^*$ with $\delta = 1/4$ and an appropriately chosen $\cF$, determined by $\cF_V$ and $\cF_E$, 
     to obtain a matching $M$ in $\cH^*$ such that every $S \in \cF$ satisfies $|S\setminus V(M)| \leq \kappa \max\{|S|, D\} / 2$.  Next, we `sparsify' $M$, by randomly and independently removing each edge with probability $\gamma$, to obtain a new matching $N$.  For each $i \in [D]$, we let $N_i \coloneqq \{e \in \cH : \exists f \in N,~f\supseteq \{i\}\times e\}$.  By Observation~\ref{obs:incidence-hypergraph}\ref{incidence-hypergraph:colouring}, the matchings $N_1, \dots, N_D$ are pairwise edge-disjoint, and the pseudorandomness property of $M$ guaranteed by Theorem~\ref{thm:egj} ensures that~\ref{pseudorandom_matching:vtx} and~\ref{pseudorandom_matching:edge} are satisfied.
\end{proof}

We will use Lemmas~\ref{lem:pseudorandom_matching} and~\ref{absorption-lemma} to construct $\lceil n / 2\rceil$ pairwise edge-disjoint matchings with nearly perfect coverage of $U$ such that the remaining edges of $\cH\setminus R$ comprise a subhypergraph of small maximum degree.  We apply the following lemma to this subhypergraph, to decompose it into matchings which are absorbable by `smallness'.  
\begin{lemma}[Leftover colouring lemma]\label{lem:leftover}
  If $\cH^\prime \subseteq \cH \setminus R$ is a spanning subhypergraph such that $\Delta(\cH^\prime) \leq \gamma n$,
  then there exist pairwise edge-disjoint matchings $N_1, \dots, N_{k}$ where $k \leq \lceil \gamma^{1/3} n\rceil$ such that 
\begin{enumerate}[label=\textbf{L\arabic*}]
    \item\label{F1} $|V(N_i)| \leq \sqrt\gamma n$ for every $i \in [k]$, and
    \item\label{F2} $\cH^\prime = \bigcup_{i=1}^k N_i$.
\end{enumerate}
\end{lemma}
\begin{proof}
  Let $D \coloneqq \lceil \gamma n\rceil$.
  For every $e\in\cH$, since $|e| \leq 3$, we have $\sum_{v\in e}d_{\cH'}(v) \leq 3D$.  Thus, $\chi'(\cH) \leq 3D + 1$, so there exist pairwise edge-disjoint matchings $M_1, \dots, M_{3D + 1}$ such that $\bigcup_{i=1}^{3D+1}M_i = \cH'$.  Let $\ell \coloneqq \lceil\gamma^{-1/2}\rceil + 1$.  For each $i \in [3D + 1]$, there exist pairwise edge-disjoint matchings $N_{i, 1}, \dots, N_{i, \ell}$ such that $\bigcup_{j=1}^{\ell}N_{i, j} = M_i$ and $|V(N_{i, j})| \leq \sqrt \gamma n$ for each $j \in [\ell]$.  By reindexing, $\bigcup_{i=1}^{3D+1}\{N_{i, 1}, \dots, N_{i, \ell}\}$ is the desired set of matchings, since $(3D + 1)\ell \leq \gamma^{1/3} n$.
\end{proof}

\subsubsection{Proof of Theorem~\ref{thm:small-efl}}

By combining Lemmas~\ref{absorption-lemma},~\ref{lem:pseudorandom_matching}, and~\ref{lem:leftover}, we prove the following lemma, which effectively reduces Theorem~\ref{thm:small-efl} to Vizing's theorem.

\begin{lemma}[Main colouring lemma]\label{lem:graph-reduction}
    There exists $\cH' \subseteq \cH$ and a proper edge-colouring of $\cH'$ using $\lceil n / 2\rceil + \lceil \gamma^{1/3}n\rceil$ colours such that 
    \begin{itemize}
        \item $\cH' \supseteq \cH\setminus R$ and
        \item the colour classes have nearly perfect coverage of $U$.
    \end{itemize}
    Moreover, $\cH\setminus \cH'$ is a graph and satisfies $\Delta(\cH\setminus\cH') \leq  n - \lceil n / 2\rceil - \lceil \gamma^{1/3}n\rceil$.
\end{lemma}
\begin{proof}
The proof proceeds in two steps.
\begin{pf-step}
    Using Lemmas~\ref{lem:pseudorandom_matching} and~\ref{absorption-lemma}, find a set $\cM$ of $\lceil n/2\rceil $ pairwise edge-disjoint matchings $M_1, \dots, M_{\lceil n / 2\rceil}$ such that the following holds:
    \begin{enumerate}[label=\textbf {M\arabic*}]
        \item\label{main-colouring:coverage} $\cM$ has nearly perfect coverage of $U$, and moreover, if $|U| < 3n/4$, then $\cM$ has perfect coverage of $U$,
        \item\label{main-colouring:leftover-degree} $\Delta(\cH\setminus (R\cup\bigcup_{i=1}^{\lceil n / 2 \rceil} M_i)) \leq \gamma n$, and
        \item\label{main-colouring:reservoir-preserved} $\Delta(R \cap \bigcup_{i=1}^{\lceil n / 2 \rceil} M_i) \leq \gamma n$.
    \end{enumerate}
\end{pf-step}
  First we partition $\cH\setminus R$ into $K\coloneqq \lceil 1 / \kappa\rceil$ pairwise edge-disjoint hypergraphs $\cH_1, \dots, \cH_K$ such that $\bigcup_{i=1}^{K} \cH_i = \cH\setminus R$, and every vertex has degree $(1/2 \pm 3\xi)n / K$ in $\cH_i$ for each $i \in [K]$.  (To show the desired partition exists, consider a partition chosen uniformly at random, which by~\eqref{non-reserved-is-regular}, will satisfy the vertex degree condition with high probability.)\COMMENT{To show that the desired partition exists, for each $e \in \cH\setminus R$, we put $e$ into exactly one of the hypergraphs $\cH_{1} , \dots , \cH_{K}$ with probability $K^{-1}$ independently at random. By~\eqref{non-reserved-is-regular}, $\mathbb{E}[d_{\cH_i}(w)] =  d_{\cH\setminus R}(w)/K = (1/2 \pm 2\xi)n/K - 1/(2K)$. Hence, by the Chernoff bound, the following holds with probability at least $1-\exp\left({-n^{1/4}}\right)$. For every $w \in V(\cH)$ and $i \in [K]$, $d_{\cH_{i}}(w) = \mathbb{E}[d_{\cH_i}(w)] \pm \xi n / (2K) = (1/2 \pm 3 \xi) n/K$.}
  We iteratively apply alternating applications of Lemmas~\ref{lem:pseudorandom_matching} and~\ref{absorption-lemma} to each $\cH_i$.
  
 Now, for each $i \in [K]$, we choose $n_{i}$ to be either $\lfloor \lceil n/2\rceil /K \rfloor$ or $\lceil \lceil n/2\rceil /K\rceil$ such that $\sum_{j=1}^{K} n_{j} = \lceil n / 2\rceil$, and we partition the set $[\lceil n / 2\rceil]$ into $K$ disjoint parts $I_1 , \dots , I_{K}$ such that $|I_i| = n_i$.  Note that $n_i \leq \kappa n$, and every vertex in $\cH_{i}$ has degree $(1 \pm 7 \xi) n_{i}$ for every $i \in [K]$.

 For $j \in [K]\cup\{0\}$, let us define the following inductive properties, where $\cM_k := \{ M_c : c\in I_k \}$ is a set of matchings in $\cH$ for each $k \in [j]$.
\begin{enumerate}[label=\textbf{M\arabic*}$_j$]
    \item\label{N1_appl} For every $k \in [j]$, $M_c \subseteq \cH_k \cup R$ for every $c \in I_k$ and moreover, the matchings in $\bigcup_{k=1}^{j} \cM_k$ are pairwise edge-disjoint.
    
    \item\label{N2_appl} For every $w \in V(\cH)$,
        \begin{align*}
      |E_{R}(w) \cap \bigcup_{k \in [j]}\bigcup_{M \in \cM_k} M| &\leq (\gamma + 3\kappa) \sum_{k \in [j]} n_k &&\textrm{ and }\\
      |E_{\bigcup_{k=1}^{j}\cH_k }(w) \setminus \bigcup_{k \in [j]}\bigcup_{M \in \cM_k} M| &\leq (\gamma + 3\kappa) \sum_{k \in [j]} n_k.
    \end{align*}

    \item\label{N3_appl}  If $|U| < 3n / 4$, then $\bigcup_{k=1}^{j} \cM_{k}$ has perfect coverage of $U$. Otherwise, $\bigcup_{k=1}^{j} \cM_{k}$ has nearly perfect coverage of $U$ with defects in $U$.
\end{enumerate}

Using induction on $j$, we will show that there exist sets of matchings $\cM_1, \dots, \cM_K$ satisfying~\ref{N1_appl}--\ref{N3_appl} for $j = K$.  
Note that~\ref{N1_appl}--\ref{N3_appl} trivially hold for $j = 0$. Let $i \in [K]$, and suppose that $\cM_1, \dots, \cM_{i-1}$ satisfy~\ref{N1_appl}--\ref{N3_appl} for $j = i-1$. Our goal is to find a collection $\cM_i$ of $n_i$ pairwise edge-disjoint matchings in $\cH$ satisfying~\ref{N1_appl}--\ref{N3_appl} for $j=i$. 

Let $R_i := R \setminus \bigcup_{k=1}^{i-1} \bigcup_{M \in \cM_k} M$, let $S_{i} := U \setminus \bigcup_{k=1}^{i-1} \bigcup_{M \in \cM_k} (U \setminus V(M))$, and let $\cW := \cF(R_i) \cup \{U,S_i\}$, where $\cF(R_i) := \{ N_{R_i}(u) \cap U \: : \: u \in U \} \cup \{ N_{R_i}(u) \setminus U \: : \: u \in U \}$.

Now we apply Lemma~\ref{lem:pseudorandom_matching} with $\cH_i$, $\cW$, $\{ E_{\cH_i}(w) \: : \: w \in V(\cH) \}$, and $n_i$ playing the roles of $\cH'$, $\cF_V$, $\cF_E$, and $ D$, respectively\COMMENT{we use the index set $I_i$ instead of $[n_i]$ for the matchings.} to obtain a set $\cN_i := \{ N_c \: : \: c \in I_i \}$ of $n_{i}$ pairwise edge-disjoint matchings in $\cH_i$ such that the following hold.

\begin{enumerate}[label=\textbf{N'\arabic*}]
    \item\label{inductive-matchings:vtx-pseudorandomness} For every $c \in I_i$, $N_c$ is $(\gamma,\kappa)$-pseudorandom with respect to $\cW$.
    
    \item\label{inductive-matchings:edge-pseudorandomness} For every $w \in V(\cH)$, $d_{\cH_i\setminus \bigcup_{c \in I_i} N_c}(w) \leq\COMMENT{$\gamma(1 + 7\xi)n_i + \kappa(1 + 7\xi)n_i \leq$}
    (\gamma + 2\kappa) n_i$.
\end{enumerate}

Now we show that for every $c \in I_i$, $(N_c , R_i , S_i)$ is absorbable by pseudorandomness of $N_c$, as follows.
\begin{itemize}
\item By~\ref{N3_appl} for $j = i-1$, if $|U| < 3n / 4$, then $S_i = U$, and otherwise $|S_{i}| \geq |U|-\sum_{k=1}^{i-1} n_{k} \geq n/4 - 1$, so~\ref{AB:S-bound} holds.
\item By~\ref{N2_appl} for $j = i-1$, since $(\gamma + 3\kappa)\lceil n / 2\rceil \leq \gamma n$,~\ref{AB:reservoir-preserved} holds.
\item By~\ref{inductive-matchings:vtx-pseudorandomness}, $N_c$ is $(\gamma ,\kappa)$-pseudorandom with respect to $\cW$  so~\ref{AB-pseudo} holds, as required.
\end{itemize}

Therefore we can apply Lemma~\ref{absorption-lemma} to obtain a set $\cM_i := \{ M_c \: : \: c \in I_i \}$ of $n_{i}$ pairwise edge-disjoint matchings in $\cH$ such that the following hold. 
\begin{itemize}
    \item For every $c \in I_i$, $M_c \supseteq N_c$ and $M_c \setminus N_c \subseteq R_i$, and consequently~\ref{N1_appl} holds for $j=i$. 
    
    \item By~\ref{inductive-matchings:edge-pseudorandomness}, for every $w \in V(\cH)$, $|E_{\cH_i}(w)\setminus \bigcup_{c\in I_i}M_c| \leq (\gamma + 2\kappa) n_i$.  Moreover, again by~\ref{inductive-matchings:edge-pseudorandomness}, all but at most $n_i - (d_{\cH_i}(w) - d_{\cH_i \setminus \bigcup_{c \in I_i} N_c}(w)) \COMMENT{$\leq 7\xi n_i + (\gamma + 2 \kappa)n_i$} \leq (\gamma + 3\kappa) n_i$ of the matchings in $\cN_i$ cover $w$, so\COMMENT{only in case $w$ is not covered by the matchings in $\cN_i'$,  then it may be covered using an edge of $R_i$} $|E_{R_i}(w) \cap \bigcup_{c \in I_i} M_c| \leq (\gamma + 3\kappa) n_i$. Hence,~\ref{N2_appl} holds for $j=i$.
    
    \item If $|U| < 3n / 4$, then $\cM_i$ has perfect coverage of $U$, and otherwise, $\cM_i$ has nearly perfect coverage of $U$ with defects in $S_i \subseteq U$. Hence,~\ref{N3_appl} holds for $j=i$.
\end{itemize}
Therefore by induction, there exist sets of matchings $\cM_1, \dots, \cM_K$ such that for every $i\in[K]$, $\cM_i$ satisfies~\ref{N1_appl}--\ref{N3_appl} for $j = i$, as claimed.  Now $\cM\coloneqq \bigcup_{i=1}^K\cM_i$ satisfies~\ref{main-colouring:coverage}--\ref{main-colouring:reservoir-preserved}.  Indeed, by~\ref{N1_appl} and~\ref{N3_appl} for $j = K$, $\cM$ satisfies~\ref{main-colouring:coverage}, and by~\ref{N2_appl} for $j = K$, $\cM$ satisfies~\ref{main-colouring:leftover-degree} and~\ref{main-colouring:reservoir-preserved}, since $(\gamma + 3\kappa)\lceil n / 2\rceil \leq \gamma n$.

\begin{pf-step}\label{step:leftover}
    Using Lemmas~\ref{lem:leftover} and~\ref{absorption-lemma}, find a set $\cM'$ of $\lceil\gamma^{1/3}n\rceil$ pairwise edge-disjoint matchings $M'_1, \dots, M'_{\lceil \gamma^{1/3}n\rceil}$ such that the following holds:
    \begin{enumerate}[label=\textbf{M'\arabic*}]
        \item\label{leftover-colouring:compatibility} $\bigcup_{M\in \cM}M\cap\bigcup_{M'\in\cM'}M' = \varnothing$, 
        \item\label{leftover-colouring:coverage} $\cM \cup \cM'$ has nearly perfect coverage of $U$, and
        \item\label{leftover-colouring:finishing} $\cH\setminus R \subseteq \bigcup_{M\in \cM}M\cup\bigcup_{M'\in\cM'}M'$.
    \end{enumerate}
\end{pf-step}
    By~\ref{main-colouring:leftover-degree} and Lemma~\ref{lem:leftover}\COMMENT{possibly adding empty matchings} applied with $\cH\setminus (R\cup \bigcup_{M\in\cM}M)$ playing the role of $\cH'$, there exists a set $\cN' \coloneqq \{N'_1, \dots, N'_{\lceil \gamma^{1/3}n\rceil}\}$ of pairwise edge-disjoint matchings such that
    \begin{enumerate}[label=\textbf{L'\arabic*}]
        \item\label{leftover-matchings:small} $|V(N'_i)| \leq \sqrt \gamma n$ for every $i \in [\lceil \gamma^{1/3}n\rceil]$ and
        \item\label{leftover-matchings:coverage} $\cH\setminus(R\cup \bigcup_{M\in\cM}M) = \bigcup_{i=1}^{\lceil\gamma^{1/3}n\rceil} N'_i$.
    \end{enumerate}
    Now we show that for every $i \in \lceil \gamma^{1/3}n\rceil$, $(N'_i, R', S')$ is absorbable by smallness of $N'_i$, where $R' \coloneqq R \setminus \bigcup_{M\in\cM}M$ and $S' \coloneqq U\setminus \bigcup_{M\in\cM}(U\setminus V(M))$.  Indeed,
    \begin{itemize}
        \item by~\ref{main-colouring:coverage}, if $|U| < 3n/4$, then $S' = U$, and otherwise $|S'| \geq |U| - \lceil n / 2\rceil \geq \gamma n$, so~\ref{AB:S-bound} holds,
        \item by~\ref{main-colouring:reservoir-preserved}, $\Delta(R\setminus R') = \Delta(R\cap \bigcup_{M\in\cM}M) \leq \gamma n$, so~\ref{AB:reservoir-preserved} holds, and
        \item by~\ref{leftover-matchings:small},~\ref{AB-pseudo} holds.
    \end{itemize}
    Therefore we can apply Lemma~\ref{absorption-lemma} to obtain a set $\cM' \coloneqq \{M'_1, \dots, M'_{\lceil \gamma^{1/3}n\rceil}\}$ of pairwise edge-disjoint matchings in $\cH$ such that the following hold:
    \begin{itemize}
        \item for every $i \in [\lceil \gamma^{1/3}n\rceil]$, $M'_i \supseteq N'_i$ and $M'_i\setminus N'_i \subseteq R'$, and
        \item $\cM'$ has nearly perfect coverage of $U$ with defects in $S'$
    \end{itemize}
    Therefore, by the choice of $R'$, $\cM'$ satisfies~\ref{leftover-colouring:compatibility}, by the choice of $S'$, $\cM'$ satisfies~\ref{leftover-colouring:coverage}, and by~\ref{leftover-matchings:coverage}, $\cM'$ satisfies~\ref{leftover-colouring:finishing}, as desired.
    
    Now let $\cH' \coloneqq \bigcup_{M\in \cM}M\cup\bigcup_{M'\in\cM'}M'$, assign colour $c$ to each edge in $M_c$ for every $c \in [\lceil n / 2\rceil]$, and assign colour $c = \lceil n/2\rceil + i$ to each edge in $M'_i$ for every $i \in [\lceil\gamma^{1/3}n\rceil$].  By~\ref{leftover-colouring:compatibility}, we have a proper edge-colouring of $\cH'$ using at most $\lceil n / 2\rceil + \lceil \gamma^{1/3} n\rceil$ colours, as required.  By~\ref{leftover-colouring:finishing}, $\cH' \supseteq \cH\setminus R$, as desired, and by~\ref{leftover-colouring:coverage}, the colour classes $\cM \cup \cM'$ of $\cH'$ have nearly perfect coverage of $U$, as desired.  Since $\cH'\supseteq \cH \setminus R$, it follows that $\cH\setminus \cH' \subseteq R$ is a graph.  Since $\cM \cup \cM'$ has nearly perfect coverage of $U$, every vertex $w \in U$ satisfies $d_{\cH \setminus \cH'}(w) \leq (n - 1) - (\lceil n / 2\rceil + \lceil\gamma^{1/3}n\rceil - 1) = n - \lceil n / 2\rceil - \lceil\gamma^{1/3}n\rceil$, 
    and by~\ref{reservoir-typicality}, every vertex $w \in V(\cH)\setminus U$ satisfies $d_{\cH\setminus \cH'}(w) \leq (1 - \eps) n / 2 + 2\xi n \leq n - \lceil n / 2\rceil - \lceil\gamma^{1/3}n\rceil$.  Hence, $\Delta(\cH\setminus \cH') \leq n - \lceil n / 2\rceil - \lceil\gamma^{1/3}n\rceil$, as desired.
\end{proof}

Now we can immediately deduce Theorem~\ref{thm:small-efl}.

\begin{proof}[Proof of Theorem~\ref{thm:small-efl}]
    By Lemma~\ref{lem:graph-reduction}, there exists $\cH'\subseteq \cH$ such that $\chi'(\cH') \leq \lceil n / 2\rceil + \lceil \gamma^{1/3}n\rceil$  and $\cH \setminus \cH'$ is a graph with $\Delta(\cH\setminus\cH') \leq n - \lceil n / 2\rceil - \lceil \gamma^{1/3}n\rceil$.  By Vizing's theorem, $\chi'(\cH\setminus \cH') \leq \Delta(\cH\setminus\cH') + 1 \leq n - \lceil n / 2\rceil - \lceil \gamma^{1/3}n\rceil + 1$.  Therefore,
    \begin{equation*}
        \chi'(\cH) \leq \chi'(\cH') + \chi'(\cH\setminus \cH') \leq n + 1,
    \end{equation*}
    as desired.
\end{proof}

We conclude this subsection by briefly discussing how Theorem~\ref{thm:small-efl} can be improved to show that $\chi'(\cH) \leq n$.  First, we note that the same argument combined with Vizing's theorem proves $\chi'(\cH) \leq n$ if at least one of the following holds:
\begin{enumerate}[(a)]
    \item\label{extra:perfect-coverage} the colour classes of $\cH'$ in Lemma~\ref{lem:graph-reduction} have perfect coverage of $U$, or
    \item\label{extra:defects} every $v\in U$ which is a `defect vertex' of some colour class of $\cH'$ in Lemma~\ref{lem:graph-reduction} satisfies $d_G(v) < n - 1$.
\end{enumerate}
Indeed, in either case, $\Delta(\cH \setminus \cH') \leq n - k - 1$ for $k = \lceil n / 2 \rceil + \lceil \gamma^{1/3} n\rceil$, and since $\chi'(\cH') \leq k$, we have $\chi'(\cH) \leq \chi'(\cH') + \chi'(\cH\setminus \cH') \leq k + (n - k - 1 + 1) = n$, as desired.  

Recall that Lemma~\ref{absorption-lemma} actually guarantees~\ref{extra:perfect-coverage} if $|U| < 3n / 4$.  In fact, this argument even works as long as $|U| \leq (1 - 10\eps)n$.  Moreover, the proof of Lemma~\ref{lem:graph-reduction} also guarantees~\ref{extra:defects} if $|\{v \in U : d_G(v) < n - 1\}| \geq (1 / 2 + 2\gamma^{1/3})n$.  In particular, with only minor modifications to the proof of Lemma~\ref{lem:graph-reduction}, we can prove either~\ref{extra:perfect-coverage} or~\ref{extra:defects} unless $U$ consists of nearly all of the vertices of $\cH$ and nearly half of the vertices of $\cH$ have degree $n - 1$.  Note that this means that $\cH$ resembles the complete graph, one of the extremal examples for the EFL conjecture.  In this case, additional ideas are needed to prove that $\chi'(\cH\setminus \cH') = \Delta(\cH\setminus \cH')$, which then ensures that $\chi'(\cH) \leq \chi'(\cH') + \chi'(\cH\setminus \cH') \leq k + (n - k) = n$, as desired.  To obtain this improved bound on $\chi'(\cH\setminus\cH')$, we modify the above approach to ensure that $\cH\setminus \cH'$ is quasirandom and then apply an edge-colouring result of Glock, K\"uhn, and Osthus~\cite{GKO2016}.  (This edge-colouring result in turn is deduced from the theorem of K\"uhn and Osthus~\cite{kuhn2013hamilton} that dense even-regular robustly expanding graphs have a Hamilton decomposition.  This deduction is based on the fact that a $\Delta$-regular graph of even order with a Hamilton decomposition has chromatic index $\Delta$.)

\subsection{Proving the EFL conjecture when all edges are large}\label{efl:large}

This subsection is devoted to the proof of Theorem~\ref{large-edge-thm}, which we restate here.

\largeEFL*

If $\ordering$ is a linear ordering of the edges of a hypergraph $\cH$, for each $e\in\cH$, we define $\fwdnbr_\cH(e)\coloneqq \{f \in N_\cH(e) : f\ordering e \}$ and $\fwddeg_\cH(e)\coloneqq |\fwdnbr_\cH(e)|$.  We omit the subscript $\cH$ when it is clear from the context. 
For each $e\in \cH$, we also let $\cH^{\ordering e} \coloneqq \{f\in \cH : f \ordering e\}$.  

For every $n$-vertex hypergraph $\cH$ and $W \subseteq \cH$, the \textit{normalized volume} of $W$ is $\vol_\cH(W) \coloneqq \sum_{e \in W}\binom{|e|}{2}\big/\binom{n}{2}$.  If $\cH$ is linear, then $\vol_\cH(W) \in [0, 1]$ for every $W \subseteq \cH$.

As in~\ref{eqn:fwd-deg-bound}, if $\cH$ is an $n$-vertex linear hypergraph such that $|e| > r$ for every $e \in \cH$, then every $e \in \cH$ satisfies $\fwddeg_\cH(e) \leq (1 + 2/r)n$, if $\ordering$ is \textit{size-monotone-decreasing} (i.e.~satisfying $e\ordering f$ if $|e| > |f|$).  In the following key lemma, we showed that we can either obtain an improved bound on $\fwddeg_\cH(e)$ (by modifying the ordering if necessary), or find a highly structured set $W \subseteq \cH$.  In particular, the edges of $W$ have similar size, and $W$ has large volume.  The fact that the edges have similar size will allow us to colour $W$ efficiently via Theorem~\ref{local-sparsity-lemma}, unless $W$ closely resembles a projective plane.
\begin{lemma}[Reordering lemma~\cite{KKKMO2021}]\label{reordering-lemma}
  Let $0 < 1/r \ll \tau, 1/K$ where $\tau < 1$, $K \geq 1$, and $1 - \tau - 7\tau^{1/4}/K > 0$.  If $\cH$ is an $n$-vertex linear hypergraph where every $e\in \cH$ satisfies $|e| \geq r$, then there exists a linear ordering $\ordering$ of the edges of $\cH$ such that at least one of the following holds.
  \begin{enumerate}[(\ref*{reordering-lemma}:a), topsep = 6pt]
  \item Every $e\in\cH$ satisfies $\fwddeg(e) \leq (1 - \tau)n$.\label{reordering-good}
  \item There is a set $W\subseteq \cH$ such that
    \begin{enumerate}[label=\textbf{W\arabic*}]
    \item\label{W-max-size} $\max_{e\in W}|e| \leq (1 + 3\tau^{1/4}K^{4})\min_{e\in W}|e|$ and
    \item\label{W-volume} $\vol_{\cH}(W) \geq \frac{(1 - \tau - 7\tau^{1/4}/K)^2}{1 + 3\tau^{1/4}K^4}$.
    \end{enumerate}
    Moreover, if $e^*$ is the last edge of $W$, then
    \begin{enumerate}[label=\textbf{O\arabic*}]
    \item for all $f\in \cH$ such that $e^*\ordering f$ and $f\neq e^*$, we have $\fwddeg(f) \leq (1 - \tau)n$ and\label{ordering-goodness}
    \item for all $e,f\in \cH$ such that $f\ordering e \ordering e^*$, we have $|f| \geq |e|$\label{ordering-by-size}.
    \end{enumerate}
    \label{reordering-volume}
  \end{enumerate}
\end{lemma}
We do not provide a proof of Lemma~\ref{reordering-lemma}, but we briefly sketch the idea.  Beginning with a size-monotone-decreasing ordering $\ordering$, we `reorder' $\ordering$ as follows.  Let $e^*$ be the last edge of $\cH$ that does not satisfy~\ref{reordering-good}.  If there exists $f \in \fwdnbr(e^*)$ such that $|N_\cH(f) \cap \cH^{\ordering e^*}| \leq (1 - \tau)n - 1$, then let $\ordering'$ be the ordering obtained from $\ordering$ by moving $f$ to be the successor of $e^*$.  If $\ordering$ satisfies~\ref{ordering-goodness} and~\ref{ordering-by-size}, then $\ordering'$ does as well, and moreover, $|\cH^{\ordering' e^*}| < |\cH^{\ordering e^*}|$.  Thus, by iterating this argument, we may assume that there is no such $f \in \fwdnbr(e^*)$.  Moreover, since we started with a size-monotone-decreasing ordering, we may also assume that $e^*$ satisfies~\ref{ordering-goodness} and~\ref{ordering-by-size}.  Now a double-counting argument shows that $W \coloneqq \{f \in \cH^{\ordering e^*} : |f| \leq (1 + 3\tau^{1/4}K)|e^*|\}$ satisfies~\ref{W-volume}. 

By applying Lemma~\ref{reordering-lemma} twice, we obtain the following.
\begin{lemma}\label{good-ordering-lemma}
  Let $0 < 1 / r \ll \sigma \ll 1$.  Let $\cH$ be an $n$-vertex linear hypergraph such that every $e \in \cH$ satisfies $|e| > r$.  If $\chi'(\cH) > (1 - \sigma)n$, then there exists a partition of $\cH$ into three spanning subhypergraphs, $\cH_1$, $W$, and $\cH_2$ such that
  \begin{enumerate}[label=\textbf{P\arabic*}]
  \item\label{W2-max-size} $\max_{e\in W}|e| \leq (1 + 4\sigma^{1/4})\min_{e\in W}|e|$,
  \item\label{W2-volume} $\vol_\cH(W) \geq 1 - 4\sigma^{1/5}$, and
  \item\label{d3-bound} $|e| \geq \max_{f\in W}|f|$ for all $e\in \cH_2$,
  \end{enumerate}
  and a linear ordering $\ordering$ of the edges of $\cH$ such that
  \begin{enumerate}[label=\textbf{FD\arabic*}]
  \item\label{reordering-goodness2} every $e\in \cH_1$ satisfies $\fwddeg_\cH(e) \leq (1 - 2\sigma)n$ and  $f \ordering e$ for every $f \in \cH_2 \cup W$, and
  \item\label{reordering-goodness3} $\fwddeg_\cH(e) \leq n/2000$ for all $e\in \cH_2$.
  \end{enumerate}
\end{lemma}

  \begin{proof}
    We apply Lemma~\ref{reordering-lemma} twice and combine the resulting orderings to obtain the desired ordering $\ordering$ of $\cH$.  
    First, we apply Lemma~\ref{reordering-lemma} to $\cH$ with $2\sigma$ and $1$ playing the roles of $\tau$ and $K$, respectively, to obtain an ordering $\ordering_1$.  
    If $\ordering_1$ satisfies~\ref{reordering-good}, then $\chi'(\cH) < (1 - \sigma)n$, so we assume~\ref{reordering-volume} holds.  Let $W$ be the set $W$ obtained from~\ref{reordering-volume}, let $e^*$ be the last edge of $W$ in $\ordering_1$, and let $\cH_1 \coloneqq \cH\setminus \cH^{\ordering_1 e^*}$.  Let $f^*$ be the edge of $W$ which comes first in $\ordering_1$, and let $\cH_2 \coloneqq \cH \setminus \{e \in \cH : f^*\ordering_1 e\}$. By the choices of $\tau$ and $K$, and since $\sigma \ll 1$, we have $\max_{e\in W}|e| \leq (1 + 4\sigma^{1/4})|e^*|$ and $\vol_\cH(W) \geq (1 - \sigma^{1/5})^3 \geq 1 - 4\sigma^{1/5}$, 
    so $W$ satisfies~\ref{W2-max-size} and~\ref{W2-volume}, as desired, and by~\ref{ordering-by-size} of~\ref{reordering-volume}, we may assume without loss of generality that every $e\in \cH$ satisfying $f^* \ordering_1 e \ordering_1 e^*$ is in $W$, so $\cH$ is partitioned into $\cH_1$, $W$, and $\cH_2$, as required, and $\cH_2$ satisfies~\ref{d3-bound}, as desired.
    
    Now we reapply Lemma~\ref{reordering-lemma} to $\cH_2$ and show that the resulting ordering satisfies~\ref{reordering-goodness2} and~\ref{reordering-goodness3}, as follows.  Apply Lemma~\ref{reordering-lemma} with $\cH_2$, $1 - 1/2000$, and $2000^2$ playing the roles of $\cH$, $\tau$, and $K$, respectively, to obtain an ordering $\ordering_2$.  Since $W\cap \cH_2 = \varnothing$, we have $\vol_\cH(W) + \vol_\cH(\cH_2) \leq 1$.  Thus, $\ordering_2$ satisfies~\ref{reordering-good}, because~\ref{reordering-volume} would imply there is a set $W'\subseteq \cH_2$ disjoint from $W$ with $\vol_{\cH}(W') >  4\sigma^{1/5}$, contradicting~\ref{W2-volume}.  
    Combine $\ordering_1$ and $\ordering_2$ to obtain an ordering $\ordering$ of $\cH$ where 
    \begin{itemize}
    \item if $f\in\cH_1 \cup W$, then $e\ordering f$ for every $e\in \cH^{\ordering_1 f}$, and
    \item if $f\in\cH_2$, then $e \ordering f$ for every $e\in \cH_2^{\ordering_2 f}$.
    \end{itemize}
    Since $\cH_1$ and $\ordering_1$ satisfy~\ref{ordering-goodness} of~\ref{reordering-volume} with $\tau = 2\sigma$,~\ref{reordering-goodness2} holds, and since $\cH_2$ and $\ordering_2$ satisfy~\ref{reordering-good} with $\tau = 1 - 1/2000$,~\ref{reordering-goodness3} holds, as desired.
  \end{proof}

To prove Theorem~\ref{large-edge-thm}, we apply Lemma~\ref{good-ordering-lemma} and consider two cases depending on the size of the edges in $W$.  In either case, by~\ref{reordering-goodness2}, it suffices to show that $\chi'(\cH_2 \cup W) \leq n$.  When the edges in $W$ have size close to or larger than $\sqrt n$, we apply the following lemma to colour $\cH_2 \cup W$.  As its proof covers the case when $\cH$ is close to a projective plane, the argument is quite delicate and we refer the reader to~\cite[Lemma~5.1]{KKKMO2021} for a proof.

\begin{lemma}\label{extremal-case-lemma}
Let $0 < 1/n_0 \ll \delta \ll 1$, and let $n \geq n_0$. If $\cH$ is an $n$-vertex linear hypergraph where every $e\in \cH$ satisfies $|e| \geq (1 - \delta)\sqrt n$, then $\chi'(\cH) \leq n$.\hfill\qed
\end{lemma}

When the edges in $W$ have size bounded away from $\sqrt n$, we apply the following lemma~\cite[Corollary~6.5]{KKKMO2021}, which we prove using Theorem~\ref{local-sparsity-lemma}, to colour $W$. 

\begin{lemma}\label{sparsity-corollary}
  Let $0 < 1/n_0, 1/r \ll \alpha \ll \zeta < 1$, let $n \geq n_0$, and suppose $r \leq (1 - \zeta)\sqrt n$.  If $\cH$ is an $n$-vertex linear hypergraph such that every $e\in\cH$ satisfies $|e| \in [r, (1 + \alpha)r]$, then $\chi'(\cH) \leq (1 - \zeta/500)n$.
\end{lemma}
\begin{proof}
  Let $\Delta \coloneqq (1 + \alpha)r(n - r)/(r - 1)$, and let $L\coloneqq L(\cH)$.  For every edge $e\in\cH$, there are at most $(1 + \alpha)r(n - r)$ pairs of vertices $\{u, v\}$ of $\cH$ where $u\notin e$ and $v\in e$.  Thus, since $\cH$ is linear and every edge has size at least $r$, we have $\Delta(L) \leq \Delta$.  Similarly, if $e, f\in\cH$ share a vertex, then $|N_L(e)\cap N_L(f)| \leq n/(r - 1) + (1 + \alpha)^2r^2 \leq (1 - 5\zeta / 6)n$.  Thus, every $v\in V(L)$ satisfies $e(L[N(v)]) \leq \Delta (1 - 5\zeta / 6)n / 2 \leq (1 - 5\zeta / 6)\binom{\Delta}{2}$.  Therefore by Theorem~\ref{local-sparsity-lemma}, $\chi'(\cH) = \chi(L) \leq (1 - 5\zeta / (6e^6))\Delta \leq (1 - \zeta / 500)n$, as desired.
\end{proof}

Now we can combine Lemmas~\ref{good-ordering-lemma}--\ref{sparsity-corollary} to prove Theorem~\ref{large-edge-thm}.

\begin{proof}[Proof of Theorem~\ref{large-edge-thm}]
  We may assume without loss of generality that $\delta \ll 1$, and we let $0 < 1 / n_0 \ll 1/r \ll \sigma \ll \delta$.  We assume $\chi'(\cH) > (1 - \sigma)n$, or else there is nothing to prove.  
  
  Apply Lemma~\ref{good-ordering-lemma} to obtain a partition of $\cH$ into $\cH_1$, $W$, and $\cH_2$ satisfying~\ref{W2-max-size}--\ref{d3-bound} and an ordering $\ordering$ of the edges of $\cH$ satisfying~\ref{reordering-goodness2} and~\ref{reordering-goodness3}, and let $r' \coloneqq \min_{e\in W}|e|$.  We assume
  \begin{equation}\label{eqn:d3-upper-bound}
    r' \leq \sqrt{n / (1 - 4\sigma)},
  \end{equation}
   as otherwise the fact that $\vol_\cH(\cH_2\cup W) \leq 1$ and~\ref{d3-bound} together would imply $e(\cH_2\cup W) \leq (1 - 2\sigma)n$\COMMENT{We have $1 \geq \vol_{\cH}(\cH_2\cup W) \geq \left.e(\cH_2 \cup W)\binom{r'}{2}\middle/\binom{n}{2}\right.$, and if $r' \geq \sqrt{n / (1 - 4\sigma)}$, then $\left.\binom{r'}{2}\middle/\binom{n}{2}\right. \geq (1 - 1/\sqrt n)/((1 - 4\sigma)n)$.}.  Together with~\ref{reordering-goodness2}, this fact would imply that every $e \in \cH$ satisfies $\fwddeg_\cH(e) \leq (1 - 2\sigma)n$, in which case $\chi'(\cH) \leq (1 - \sigma)n$, a contradiction.
  
  We now consider two cases: $r' < (1 - \delta)\sqrt n$, and $r' \geq (1 - \delta)\sqrt n$.  In the former case, we derive a contradiction by showing $\chi'(\cH) \leq (1 - \sigma)n$, and in the latter case, we prove that $\chi'(\cH) \leq n$ and $|\{e\in\cH : |e| = (1\pm\delta)\sqrt n\}| \geq (1 - \delta)n$.
  
  \noindent {\bf Case 1:} $r' < (1 - \delta)\sqrt n$.
  
  Let $\zeta \coloneqq 1 - r' / \sqrt n$.  Since $r' < (1 - \delta)\sqrt n$, we have $\zeta > \delta$.  
  By~\ref{W2-max-size} and Lemma~\ref{sparsity-corollary} with $r'$ and $4\sigma^{1/4}$ playing the roles of $r$ and $\alpha$, respectively, we have $\chi'(W) \leq (1 - \zeta / 500)n$.
  
  Now we claim that $\chi'(\cH_2) \leq \zeta n / 1000$.  To that end, let $k \coloneqq e(\cH_2)$.  If $k \leq \zeta n/ 1000$, then we can simply assign each edge of  $\cH_2$ a distinct colour and the claim holds, so we assume $k > \zeta n/ 1000$.  Since $\zeta > \delta$, we have $k > \zeta n/ 1000 > 2\delta^2 n$.  By~\ref{d3-bound}, every edge of $\cH_2$ has size at least $r'$, so we have $\vol_{\cH}(\cH_2) \geq k(r' - 1)^2 / n^2$.  On the other hand, by~\ref{W2-volume}, and since $\cH_2 \cap W = \varnothing$, we have $\vol_{\cH}(\cH_2) \leq 4\sigma^{1/5} \leq \delta^3$.  Thus, $2\delta^2 n < k \leq \delta^3 n^2 / (r' - 1)^2$, so $r' < \delta^{1/4} \sqrt{n}$.
  Therefore, $\zeta > 1000/1001$.  Now by~\ref{reordering-goodness3}, we can properly colour $\cH_2$ greedily in the ordering provided by $\ordering$ using at most $n/2000 + 1 \leq \zeta n/ 1000$ colours, as claimed.

  Since $\chi'(\cH_2) \leq \zeta n/ 1000$ and $\chi'(W) \leq (1 - \zeta / 500)n$, there is a proper edge-colouring of $\cH_2\cup W$ using at most $(1 - \zeta / 1000)n \leq (1 - \sigma)n$ colours\COMMENT{since $\sigma \ll \delta < \zeta$}, and by~\ref{reordering-goodness2}, we can extend such a colouring to $\cH_1$ greedily without using any additional colours, contradicting that $\chi'(\cH) > (1 - \sigma)n$.

  \noindent {\bf Case 2:} $r' \geq (1 - \delta)\sqrt n$.
  
  By~\ref{d3-bound} and Lemma~\ref{extremal-case-lemma}, there is a proper edge-colouring of $\cH_2\cup W$ using at most $n$ colours, and as before, by~\ref{reordering-goodness2}, we can extend such a colouring to $\cH_1$ greedily without using any additional colours.  Hence, $\chi'(\cH) \leq n$, as desired.

  Since $r' \geq (1 - \delta)\sqrt n$, by~\ref{W2-max-size} and~\eqref{eqn:d3-upper-bound}, the edges in $W$ have size $(1\pm \delta)\sqrt n$.  In fact, the edges in $W$ have size at most $(1 + \delta^2)\sqrt n$, so by~\ref{W2-volume}, since $\vol_{\cH}(W) \geq 1 - \delta^2$, we have $e(W) \geq \vol_\cH(W)(n - 1) / (1 + \delta^2)^2 \geq (1 - \delta)n$\COMMENT{since $\vol_\cH(W) \leq e(W)\left.\binom{(1 + \delta^2)\sqrt n}{2}\middle / \binom{n}{2}\right. \leq e(W) (1 + \delta^2)^2 / (n - 1)$}, as desired.
\end{proof}

We conclude by briefly discussing how to combine the arguments of Theorems~\ref{thm:small-efl} and~\ref{large-edge-thm} to obtain Theorem~\ref{thm:efl}.  First, we merge the hierarchy used in the proof of Theorem~\ref{large-edge-thm} with~\eqref{eqn:small-efl-hierarchy} and also introduce constants $r_1$, $r_0$, $\beta$, and $\rho$ into the hierarchy, letting
\begin{equation*}
    1 / n_0 \ll 1 / r_0 \ll \xi \ll 1 / r_1 \ll \beta \ll \kappa \ll \gamma \ll \eps \ll \rho \ll \sigma \ll \delta \ll 1.
\end{equation*}
As in the proof of Theorem~\ref{thm:asym-list-bf}, we decompose $\cH$ into three spanning subhypergraphs $\cH_{\mathrm{sml}} \coloneqq \{ e \in \cH \: : \: |e| \leq r_1 \}$, $\cH_{\mathrm{med}} \coloneqq \{ e \in \cH \: : \: r_1 < |e| \leq r_0 \}$, and $\cH_{\mathrm{lrg}} \coloneqq \{ e \in \cH \: : \: |e| > r_0 \}$.  We apply a stronger version of Theorem~\ref{large-edge-thm} to $\cH_{\mathrm{lrg}}\cup\cH_{\mathrm{med}}$ in which 
\begin{enumerate}[(a)]
\item\label{extra:bounded-color-classes} every colour class either covers at most $\beta n$ vertices or consists of a single edge, and
\item\label{extra:medium-edges} at most $\gamma n$ colours are used to colour $\cH_{\mathrm{med}}$. 
\end{enumerate}
This strengthening of Theorem~\ref{large-edge-thm} enables us to modify the proof of Lemma~\ref{lem:graph-reduction} to find a colouring of some $\cH'$ satisfying $\cH_{\mathrm{sml}}\setminus R \subseteq \cH' \subseteq \cH_{\mathrm{sml}}$ compatible with the colouring of $\cH_{\mathrm{lrg}}\cup\cH_{\mathrm{med}}$.  As in the proof of Theorem~\ref{thm:small-efl}, we can ensure that $\cH\setminus \cH'$ is a graph of maximum degree at most $n - \lceil n / 2\rceil - \lceil \gamma^{1/3}n\rceil$; however, Vizing's theorem does not guarantee a colouring of $\cH\setminus \cH'$ that avoids conflicts with $\cH_{\mathrm{lrg}}\cup\cH_{\mathrm{med}}$.  To that end, we modify Lemma~\ref{lem:graph-reduction} further to colour $\cH'$ with $k \coloneqq \lceil (1 - \rho)n \rceil$ colours.  If $\cH_{\mathrm{lrg}}\cup\cH_{\mathrm{med}}$ can be coloured with at most $(1 - \sigma)n$ colours, then we can colour $\cH\setminus \cH'$ with $n - k$ colours that are not used on  $\cH_{\mathrm{lrg}}\cup\cH_{\mathrm{med}}$ (either using Vizing's theorem or the more involved argument discussed at the end of Section~\ref{efl:small}).  If $\cH_{\mathrm{lrg}}\cup\cH_{\mathrm{med}}$ requires more than $(1 - \sigma) n$ colours, then we need a different approach, using the fact that in this case we know that $|\{e \in \cH : |e| = (1 \pm \delta)\sqrt n\}| \geq (1 - \delta)n$, i.e.~that $\cH$ is close to a projective plane.  See~\cite{KKKMO2021} for the full proof.

\section*{Acknowledgements}

We are grateful to Jeff Kahn for helpful comments on a preliminary version of the survey and to Dhruv Mubayi for informing us that the conjecture of Frieze and Mubayi~\cite{FM13} listed in a previous version of the survey was disproved in~\cite{CM2017}.  We also thank an anonymous reviewer for their helpful comments.

\bibliographystyle{amsabbrv}
\bibliography{ref}

\end{document}